\documentclass[10pt,twoside,a4paper,reqno]{amsart}
\usepackage{a4wide}
\usepackage{amsmath, amsfonts, amsthm, amssymb, verbatim}
\usepackage[mathcal]{euscript}
\usepackage{graphicx}
\usepackage{caption,subfig} 
\usepackage{dsfont}
\usepackage{tikz}
\newtheorem{theorem}{Theorem}[section]
\newtheorem{proposition}[theorem]{Proposition}
\newtheorem{corollary}[theorem]{Corollary}
\newtheorem{remark}[theorem]{\it Remark\/}
\newtheorem{definition}[theorem]{Definition}
\newtheorem{lemma}[theorem]{Lemma}
\newtheorem{assumption}[theorem]{Assumption}
\numberwithin{equation}{section}
\numberwithin{figure}{section}

\DeclareMathOperator{\loc}{loc}
\DeclareMathOperator{\sgn}{sgn}
\newcommand \RR		  {\mathbb{R}}
\newcommand \del		\partial
\newcommand \eps		\epsilon
\newcommand \calC		{\mathcal{C}}
\newcommand \calU		{\mathcal{U}}
\newcommand \calV		{\mathcal{V}}
\newcommand \calL		{\mathcal{L}}

\newcommand \indic	{\mathds{1}}
\newcommand \argmin	{\mathrm{argmin}}
\newcommand \supp		{\mathrm{supp}}

\begin{document} 
\title[Coupling techniques for nonlinear hyperbolic equations]{Coupling techniques for nonlinear hyperbolic equations. II. Resonant interfaces with internal structure}

\author[Benjamin Boutin, Fr\'ed\'eric Coquel, and Philippe G. L{\tiny e}Floch]
{Benjamin Boutin$^1$, Fr\'ed\'eric Coquel$^2$, and Philippe G. L{\smaller e}Floch$^3$}

\address{Univ. Rennes, CNRS, IRMAR - UMR 6625, F-35000 Rennes, France}
\email{Benjamin.Boutin@univ-rennes1.fr}

\address{Centre de Math\'ematique Appliqu\'ees, Ecole Polytechnique, Route de Saclay, 91128 Palaiseau Cedex, France}
\email{coquel@cmap.polytechnique.fr}

\address{Laboratoire Jacques-Louis Lions \& Centre National de la Recherche Scientifique, Sorbonne Universit\'e, 75258 Paris, France}
\email{contact@philippelefloch.org}

\date{}

\thanks{The authors were partially supported by the Innovative Training Network (ITN) grant 642768 (ModCompShock).
\noindent\textit{Keywords.} Hyperbolic conservation law; coupling technique; Riemann problem; self-similar approximation; resonant effect. \hfill To appear in: \it Networks and Heterogeneous Media.
\\
\phantom{} \hfill Final version: {January 2021}}

\begin{abstract} 
In the first part of this series, an augmented PDE system was
introduced in order to couple two nonlinear hyperbolic equations together. This
formulation allowed the authors, based on Dafermos's self-similar viscosity method,
to establish the existence of self-similar solutions to the coupled Riemann
problem. We continue here this analysis in the {restricted} case of
one-dimensional scalar equations and investigate the internal structure of the
interface in order to derive a selection criterion associated with the underlying regularization
mechanism and, in turn, to characterize the nonconservative interface layer. In addition, we identify a new criterion 
that selects double-waved solutions that are also continuous at the interface. We conclude by providing some
evidence that such solutions can be non-unique when dealing with non-convex flux-functions.
\end{abstract}
  
\maketitle


\section{Introduction}


\subsection{Motivation and notations}

In the series of papers \cite{Boutin:2011la, BCL-3, BCL-4}, the authors developed analytical and numerical techniques for the coupling of nonlinear hyperbolic systems across a fixed spatial interface. We continue our investigations and study 
Dafermos's self-similar, viscosity approximation for an augmented system. We consider here the coupling between
 two conservation laws on one spatial dimension, that is, 
\begin{align}
\label{eq:left}
  \del_t u + \del_x f_-(u) &= 0,\qquad x<0,\ t>0,\\
\label{eq:right}
  \del_t u + \del_x f_+(u) &= 0,\qquad x>0,\ t>0.
\end{align}
A coupling condition at the interface $\{x=0\}$ should involve the (formal) traces of the unknown function $u = u(x,t)$ from
 both sides, and we focus here on a \emph{nonconservative coupling condition}, obtained by imposing the continuity condition 
  (whenever it is compatible with the dynamics) 
\begin{equation}
\label{eq:cplcond}
u(t,0^-) = u(t,0^+),\quad t>0.
\end{equation}
Roughly speaking, this ensures that any constant function is an equilibrium solution, independent of time. In other words, 
there is no evolution when the traces of $u$ coincide despite the flux terms from both sides may differ, that is,
in general $f_-(u(t,0^-)) \neq f_+(u(t,0^+))$. Importantly, the usual Rankine-Hugoniot condition, in general, is not satisfied across the interface.
{This feature is present in some systems of conservation laws, typically when handling point constraints or junctions. It leads to well-posedness issues that typically require the introduction of adapted Riemann solvers, see e.g. \cite{Benyahia:2018,Corli:2018,Garavello:2011,Garavello:2006,Herty:2007,Rosini:2020}}

The continuity of the flux at the interface is a more common interfacial coupling condition. 
We will not treat this case here but for the sake of comparison, let us briefly discuss it first. 
The continuity of flux terms across the interface in fact corresponds to the usual Rankine-Hugoniot relation and is sufficient 
in order to characterize conservative, standing wave, interfacial discontinuities. 
The two half-problems may be reformulated as a single global conservation law, with an inhomogeneous flux function $f=f(x,u)$, that is discontinuous at $x=0$. A large literature is devoted to these \emph{conservative coupling conditions}, as we call them. 
Various approaches have been introduced in order to tackle the corresponding Cauchy problem
(after Kru{\v z}kov~\cite{Kruzkov:1970kq}) and the well-posedness is obtained 
by relying on suitable entropy-like inequalities at the interface. 
We refer for example to the $\Gamma$-condition and the viscous profile condition introduced by Diehl~\cite{Diehl:1996ly,Diehl:1996gf}, to the results of Klingenberg and Risebro~\cite{Klingenberg:1995xy} in the case of multiplicative fluxes $g(x)f(u)$, Seguin and Vovelle~\cite{Seguin:2003bh}, Audusse and Perthame~\cite{Audusse:2005fy}, and a review on the subject by Burger and Karlsen~\cite{Burger:2008oq}. 
Adimurthi, Mishra, and Gowda~\cite{Adimurthi:2005cr,Adimurthi:2007wd,Adimurthi:2007eu} developed the notion of generalized entropy solution for such discontinuous flux functions. Later on, a unified framework was provided by the $L^1$-dissipative admissibility germ developed by Andreianov, Karlsen and Risebro~\cite{Andreianov:2010ai, Andreianov:2011tg} on the basis of $L^1$-contractive semigroups involving boundary conditions (see also more recent works~\cite{Andreianov:2012dp, Andreianov:2014hc, Andreianov:2015sp, Andreianov:2015th, Andreianov:2015ij,Andreianov:2012wo} and the references therein).

Motivated by the theory of admissible boundary conditions for hyperbolic systems of conservation laws, as derived by Bardos, Leroux and Nedelec~\cite{Bardos:1979gf} and, more generally, defined in Dubois and LeFloch~\cite{Dubois:1988ly}, 
the nonconservative coupling problem under consideration in the present paper
was studied at, both, the theoretical and at the numerical levels. 
Godlewski and Raviart~\cite{Godlewski:2004qf} analyzed the scalar problem first and then treated systems from plasma fluid dynamics, in collaboration with Le Thanh~\cite{Godlewski:2005jk}. These authors proposed adapted numerical strategies based on generic numerical fluxes chosen in agreement with the theoretical results, whenever a unique solution is available. 
When a suitable sign of a characteristic field changes across the coupling interface, uniqueness may fail and different numerical schemes are found to capture different solutions; see also~\cite{Ambroso:2014ij, Boutin:2010fv,Chalons:2010fu, galie:tel}. This non-uniqueness and numerical selection phenomenon is a well-known feature of general nonconservative problems, as recognized by \cite{HouPLF}
and investigated in \cite{Berthon:2012,Castro:2008}.
The theory of nonconservative paths proposed by LeFloch \cite{LeFloch-IMA,LeFloch-graph} and Dal~Maso, LeFloch and Murat \cite{Dal-Maso:1995cs} provides at once an explanation and a cure to such a difficulty. Namely, the presence of a nonconservative product requires some new additional characterization to be well-defined in the sense of measures. The 
notion of {DLM product\footnote{{DLM stands for the initials of the authors of \cite{Dal-Maso:1995cs}.}}} provides one with a framework to define and handle entropy weak solutions. It is expected that a given numerical scheme, due to its high-order (diffusive-dispersive) features, generates a family of paths and, therefore, selects one solution among all possible {ones}; for further details see \cite{LeFloch-IMA} as well as \cite{Berthon:2012,Castro:2008}. {
Another intermediate coupling strategy consists in introducing a weighted Dirac mass source, and enables the coupling conditions to take into account the mass, momentum, or energy loss at the discontinuity \cite{Coquel:2016, galie:tel}.}

\medskip

Following {the previous paper in this series} \cite{Boutin:2011la}, the coupling problem~\eqref{eq:left}--\eqref{eq:right}--\eqref{eq:cplcond} is now 
studied by introducing a reformulation based on the nonconservative extended PDE system
\begin{equation}
  \label{eq:extended-eq}
    \begin{aligned}
      \partial_tu+\dfrac12\left((1-v)f_-'(u)+(1+v)f_+'(u)\right)\partial_x u&=0,\\
      \partial_tv&=0, 
    \end{aligned}
\end{equation}
together with the Riemann initial conditions
\begin{equation}
\label{eq:Riemanndata}
    (u,v)(x,0)=\begin{cases}
        (u_L,-1),& x<0,\\
        (u_R,+1),& x>0.
      \end{cases}
\end{equation}
{For technical reasons, we do not attempt to treat systems of equations and consider hereafter either nonlinear conservation laws or special coupled systems that essentially reduce to scalar problems~\eqref{eq:extended-eq}).}
As mentioned above, the nonconservative coupling condition~\eqref{eq:cplcond} may be handled by gluing together 
two half-problems and formulating a {\sl set of admissible traces} in the sense of~\cite{Dubois:1988ly}, {as used} in~\cite{Boutin:2010fv}.
The ``artificial color function'' $v$ admits a standing discontinuity at $x=0$ so that, at least formally, we 
recover the hyperbolic conservation laws~\eqref{eq:left}--\eqref{eq:right} in each half-domain. However,
the system~\eqref{eq:extended-eq} exhibits a loss of strict hyperbolicity when the wave velocities have opposite 
signs across the interface. In that case, for some intermediate value $v\in[-1,1]$, the characteristic velocity 
$\tfrac12\left((1-v)f_-'(u)+(1+v)f_+'(u)\right)$ vanishes and is equal to the velocity of the stationary (and thus 
linearly degenerate) wave. 
The reader is referred to the works~\cite{Chalons:2008mi,Goatin:2004la,Isaacson:1992ff} for similar discussions. 
As a matter of fact, we propose here to tackle the problem directly in a global-in-space form. We then consider a 
suitable self-similar viscous approximation 
in order to be in a position to study the layer near the interface and eventually uncover
the underlying selection criterion. In turn, this criterion corresponds to an (implicit) definition of a DLM path 
adapted to the problem.

Dafermos's approximations to \eqref{eq:extended-eq} are self-similar solutions
(that is, solutions depending only on $\xi=x/t$) 
of a vanishing viscous formulation involving a time-dependent, suitably scaled, parabolic regularization. Such
a method was proposed independently by Kala{\v s}nikov \cite{Kalasnikov:1959ud}, Tup{\v c}iev \cite{Tupciev:1966kk},
and Dafermos \cite{Dafermos:1973fv}. More recent works by Christoforou and Spinolo \cite{Christoforou:2012kn} aim 
at reducing the gap between such a self-similar regularization (on the ODE viscous profile and/or Riemann hand 
side) and the more physical parabolic regularization (on the PDE and/or Cauchy problem hand side). This kind of 
reformulation has also recently been used by Berthon~et al. in~\cite{Berthon:2019} to study diffusive-dispersive 
solutions and identify numerically nonclassical undercompressive weak solutions, another transition problem 
involving kinetic relations.

After all, the system under study, consisting in the self-similar Dafermos approximation
of~\eqref{eq:extended-eq}, reads as follows:
\begin{equation}
\label{eq:Dafermos}
\begin{aligned}
-\xi \dfrac{du^\eps}{d\xi} + \dfrac12\left((1-v^\eps)f_-'(u^\eps)+(1+v^\eps)
 f_+'(u^\eps)\right)\dfrac{du^\eps}{d\xi} &= \eps \dfrac{d^2u^\eps}{d\xi^2},\\
 -\xi \dfrac{dv^\eps}{d\xi} &= \eps^2 \dfrac{d^2v^\eps}{d\xi^2}.
\end{aligned}
\end{equation}

Throughout this paper, we impose the following regularity. 

\begin{assumption}
The flux functions $f_-$ and $f_+$ are twice continuously differentiable over $\RR$ and admit a finite number of inflection points.
\end{assumption}

Consider Riemann data $(u_L,u_R)$ as in~\eqref{eq:Riemanndata}, and let $M>0$ be the following upper bound for the 
characteristic speeds involved in the problem (with $u\in[u_L,u_R]$, independently of $v\in[-1,1]$):
\begin{equation}
\label{eq:hyp-bound}
M = \sup_{u\in[u_L,u_R]}\left(|f_-'(u)|+|f_+'(u)|\right).
\end{equation}
The system of ODEs~\eqref{eq:Dafermos} is then supplemented with natural boundary conditions at $\xi=\pm M$, and
 reproduce the initial Riemann data~\eqref{eq:Riemanndata} at the level of self-similar solutions:
\begin{equation}
\label{eq:Boundcond}
\begin{aligned}
u^\eps(-M)=u_L, & \qquad v^\eps(-M)=-1,\\ 
u^\eps(+M)=u_R, & \qquad  v^\eps(+M)=+1.
\end{aligned}
\end{equation}
We are interested in the qualitative properties of the solutions $(u^\eps,v^\eps)$ to~\eqref{eq:Dafermos}-\eqref{eq:Boundcond}, as the viscosity parameter $\eps$ goes to zero.


\subsection{Main results and structure of this paper}

Let us first recall some results of our previous paper \cite{Boutin:2011la} about the existence of a solution $u^\eps$ to \eqref{eq:Dafermos}-\eqref{eq:Boundcond} and the convergence of this solution as $\epsilon \to 0$. These results are summarized and reduced to the current framework in Proposition~\ref{prop:MR2} below.
{Actually, our previous study covered a large class of systems of conservation laws but, as already mentioned, we restrict ourselves here to} scalar conservation laws \eqref{eq:extended-eq}. This one-dimensional scalar framework under consideration makes it possible to further analyze the structure of the solutions, thanks to monotonicity arguments and representation formulas. The aim of this paper is indeed to {‘\sl characterize the limiting objects.} 

\begin{definition}[CRD solution]
Given $(u_L,u_R)\in\RR^2$, a function $u\in L^\infty(\RR)\cap BV(\RR)$ is said to be a \emph{\bf coupled Riemann-Dafermos solution} 
(or CRD solution, in short)
{to \eqref{eq:left}--\eqref{eq:right}, }
associated with the Riemann data $(u_L,u_R)$,
if it is the limit of a convergent subsequence $\{u^\eps\}_{\eps>0}$ of solutions to the boundary value problem~\eqref{eq:Dafermos}--\eqref{eq:Boundcond}.
\end{definition}

We rely on the following implicit representation formula for the solution $u^\eps$  to \eqref{eq:Dafermos}--\eqref{eq:Boundcond}:
\begin{equation}
\label{eq:repres}
 u^\eps(\xi) = u_L + (u_R-u_L)\dfrac{\displaystyle\int_{-M}^\xi e^{-h^\eps(s;u^\eps)/\eps} ds}{\displaystyle\int_{-M}^M e^{-h^\eps(s;u^\eps)/\eps} ds}, \quad -M\leq \xi \leq M,
\end{equation}
where  
\begin{equation}
\label{eq:heps}
h^\eps (\xi;u^\eps) = \frac12\int_0^\xi \Big(2s-f_-'(u^\eps(s))(1-v^\eps(s))-f_+'(u^\eps(s))(1-v^\eps(s))\Big)ds. 
\end{equation}
The following results are available. 

\begin{proposition}[{\cite[Theorem 3.5]{Boutin:2011la}}]
\label{prop:MR2}
Let $(u_L,u_R)$ and $M$ be given so that~\eqref{eq:hyp-bound} holds. 
Any CRD solution $u$ is monotone and bounded in $(-M,M)$. Moreover, $u$ is solution of the following conservation laws endowed with infinitely many entropy inequalities (any of the following equations being satisfied in the sense of distributions over the respective half-space $(0,M)$ or $(-M,0)$): 
\begin{equation}
\label{eq:ESleft}
-\xi \dfrac{d}{d\xi} u + \dfrac{d}{d\xi} f_-(u) = 0,\qquad
-\xi \dfrac{d}{d\xi} \eta(u) + \dfrac{d}{d\xi} q_-(u) \leq 0, \qquad \xi<0,
\end{equation}
\begin{equation}
\label{eq:ESright}
-\xi \dfrac{d}{d\xi} u + \dfrac{d}{d\xi} f_+(u) = 0,\qquad 
-\xi \dfrac{d}{d\xi} \eta(u) + \dfrac{d}{d\xi} q_+(u) \leq 0, \qquad \xi>0,
\end{equation}
together with the boundary conditions
\begin{equation}
\label{eq:BC}
u(-M) = u_L,\qquad u(+M) = u_R.
\end{equation}
{The above convex entropy-entropy flux pairs $(\eta,q_-)$ and $(\eta,q_+)$ are associated with the flux $f_-$ and $f_+$ under consideration, respectively, in the sense that $q_\pm ' = \eta' f_\pm '$.}
If $u_-$ and $u_+$ denote the traces of $u$ at $\xi=0-$ and $\xi=0+$, respectively, one of the following cases must hold: 
\begin{equation}
\label{eq:monotonicity}
\begin{aligned}
\textrm{either }&u_L\leq u_- \leq u_+ \leq u_R,\\
\textrm{or }&u_L\geq u_- \geq u_+ \geq u_R.
\end{aligned}
\end{equation}
\end{proposition}

In other words, any CRD solution $u$ coincides with some entropy weak solution for the left flux $f_-$ in the half-line $\{\xi<0\}$ connecting $u_L$ to $u_-$ and with an entropy weak solution for the right flux $f_+$ in $\{\xi>0\}$ connecting $u_+$ to $u_R$. As usual, the compound self-similar solution in the half line $\{\xi<0\}$ is recovered when considering respectively the lower convex envelope of the flux function $f_-$ in $[u_L,u_-]$ when $u_L<u_-$ (respectively the upper concave envelope of $f_-$ in $[u_-,u_L]$ otherwise). The compound solution in the right half line $\{\xi>0\}$ is obtained when considering exactly the same kind of envelope but for the flux $f_+$ in $[u_+,u_R]$.

It is important to observe, at this level, that there is no reason why such a CRD solution would be unique (the Riemann data being given). Both equations \eqref{eq:ESleft} and \eqref{eq:ESright} concern only one half-space and, except the simple monotonicity property \eqref{eq:monotonicity}, no precise information has been yet obtained at the interface $\xi=0$ to connect them together. At this stage however, the monotonicity characterization is sufficient to state that for the trivial Riemann data $(u_L,u_R)$ with $u_L=u_R$, the unique CRD solution is the constant one.

This paper is organized as follows. Section~\ref{sec:outer} is devoted to the analysis of the outer solution. By studying the  asymptotic behavior of $u^{\eps}$ far from the coupling interface, we get necessary conditions on a function to be a CRD solution, supplementing in this way Proposition~\ref{prop:MR2} above. 
The new selection criterion appears in Proposition~\ref{prop:MR3} and is specified in Corollary~\ref{cor:convex} for the case of convex fluxes. In Section~\ref{sec:layer}, we characterize the asymptotic interfacial layer in the neighborhood of $\xi=0$ in terms of a possible viscous profile equation. This is the object of Theorem~\ref{thm:RelaxLayer}. In Section~\ref{sec:matching} the analysis is concerned with the matching conditions at the interface. This is the object of Theorem~\ref{thm:matchCond}.
{Our results provide necessary conditions inherited from the viscous self-similar approximation. However, they might still be insufficient in order to guarantee that the constructed solution is an actual limit of viscous solutions $u^\eps$.} Section~\ref{sec:GNL} is devoted to some more explicit results in the case of convex fluxes and then to the convex quadratic case. Finally, some numerical prospecting around our new selection criterion are given in Section~\ref{sec:numerics}.


\section{Analysis of the outer solution}
\label{sec:outer}

From Proposition~\ref{prop:MR2}, any CRD solution exhibits a global monotonicity property over the whole real line, and {coincides} with an entropy weak solution over each half-space. The following proposition concerns another global interesting feature of the CRD solutions. It extends a natural structure put forward by Tzavaras~\cite{Tzavaras:1996kl} for systems and, {while} analyzing characteristic boundary layers, by Joseph and LeFloch~\cite{Joseph:1999dq,Joseph:2002dq,Joseph:2006dq}.

\begin{proposition}
\label{prop:MR3}
Let $u$ be a CRD solution to \eqref{eq:left}--\eqref{eq:right}. Let us define for any $\xi\in[-M,M]$ the quantity
\begin{equation}
\label{eq:h}
h(\xi;u)=\int_0^\xi \left(s- f'_-(u(s)) \indic_{s<0} - f'_+(u(s)) \indic_{s>0}\right)\,ds.
\end{equation}
Then, the support of the measure $u'$ coincides with the global minimizing set of the function $h(\cdot;u)$ over the interval $[-M,M]$:
\begin{equation}
\label{eq:MR3}
\supp\ u'=\argmin_{\xi} h(\xi;u).
\end{equation}
\end{proposition}

This has to be understood as a feedback condition on $u$ to be effectively a possible CRD solution, that is, 
a limit of $\{u^\eps\}_{\eps>0}$ in the considered vanishing viscosity formulation. The equation~\eqref{eq:MR3} expresses thus a necessary condition involving both left and right wave fans of the CRD solution $u$ as a whole.
The next Corollary describes the typical configurations that may arise, according to the nature of waves in both half-problems.

\begin{corollary}
\label{cor:MR3}
Let $u$ be a CRD solution to \eqref{eq:left}--\eqref{eq:right} and consider then its left-hand and right-hand traces $u_-$ and $u_+$ nearby the coupling point $\xi=0$. The following facts are satisfied:
\begin{enumerate}
\renewcommand{\theenumi}{\roman{enumi}}
\renewcommand{\labelenumi}{(\theenumi)}
\item If $f'_+(u_+)<0$, then $\supp\,u'\cap  (0,M)=\emptyset$, i.e. $u(\xi)=u_R$, $\xi>0$.\\
Similarly, if $f'_-(u_-)>0$, then $\supp\,u'\cap  (-M,0)=\emptyset$, i.e. $u(\xi)=u_L$, $\xi<0$.
\smallskip
\item\label{cor:main3point2}
If $\supp\,u'\cap (-M,0)\neq\emptyset$ and $\supp\,u'\cap (0,M)\neq\emptyset$,
then $f_-'(u_-)\leq 0$ and $f_+'(u_+)\geq 0$. Moreover, one has
\begin{equation}
\label{eq:selection}
\min_{\xi\leq 0}h(\xi;u) = \min_{\xi\geq 0}h(\xi;u).
\end{equation}
\item Assume $u_-\neq u_+$.\\
If $\supp\,u'\cap (0,M)\neq \emptyset$ then $f_+'(u_+)=0$.\\
Similarly, if $\supp\,u'\cap (-M,0)\neq \emptyset$ then $f_-'(u_-)=0$.
\end{enumerate}
\end{corollary}

Before proving Proposition~\ref{prop:MR3}, let us introduce the following straightforward Lemma, describing the general form of the function $h(\cdot;u)$.

Let us first introduce some notations.
Consider any given CRD solution $u$ to \eqref{eq:left}--\eqref{eq:right}. 
From the preliminary Proposition~\ref{prop:MR2} and from the classical construction of admissible wave fans for scalar hyperbolic conservation laws (see Dafermos~\cite{Dafermos:2010dk}), the following properties are well-known.
The intervals $(-M,0)$ and $(0,M)$ may be decomposed into the union of three pairwise disjoint sets
$
(-M,0)=\mathcal{C}_-\cup\mathcal{S}_-\cup\mathcal{W}_-\textrm{ and }(0,M)=\mathcal{C}_+\cup\mathcal{S}_+\cup\mathcal{W}_+,
$
where we denote
\begin{itemize}
\item $\mathcal{C}_-=(-M,0)\setminus\supp\, u'$ and $\mathcal{C}_+=(0,M)\setminus\supp\, u'$.
\item $\mathcal{S}_-\cup\mathcal{S}_+$ is the (at most) countable set of points of discontinuity of $u$, across each of which the Rankine-Hugoniot jump relation and the entropy admissibility inequalities of Liu-Oleinik are satisfied (either for the flux $f_-$ over $\mathcal{S}_-$, or for $f_+$ over $\mathcal{S}_+$).
\item $\mathcal{W}_-\cup\mathcal{W}_+$ is the (possibly empty) set of points of continuity of $u$ that lie in $\supp\, u'\setminus\{0\}$.
\end{itemize}
Moreover, $u$ is discontinuous at $\xi=0$ if and only if $u_-\neq u_+$. We recall at this point that the Rankine-Hugoniot relation does not {a priori apply}.
{Let us introduce} the following bounds for the negative and positive waves:
\begin{equation}
\label{eq:boundwaves}
\begin{aligned}
&\Lambda_- := \inf \left(\mathcal{S}_-\cup\mathcal{W}_-\right),\quad & \lambda_-:=\sup \left(\mathcal{S}_-\cup\mathcal{W}_-\right),\\
&\lambda_+ := \inf \left(\mathcal{S}_+\cup\mathcal{W}_+\right),& \Lambda_+:=\sup \left(\mathcal{S}_+\cup\mathcal{W}_+\right).
\end{aligned}
\end{equation}
To deal with more degenerate situations, we set $\lambda_-=\Lambda_-:=0$ if $\mathcal{S}_-\cup\mathcal{W}_-=\emptyset$, and $\lambda_+=\Lambda_+:=0$ if $\mathcal{S}_+\cup\mathcal{W}_+=\emptyset$. 
Finally to handle with empty intervals, we consider the following convention: if $\mathcal{S}_-\cup\mathcal{W}_-=\{a\}$ for some $a<0$, we then set $\Lambda_-=\lambda_-=a$ and $(a,a)=\emptyset$.\\
As a consequence of the Lipschitz continuity of $f_-'$ and $f_+'$, the interior sets
\begin{equation}
\mathcal{H}_-:=\mathrm{int}\left(\mathcal{S}_-\cup\mathcal{W}_-\right)\textrm{ and }\mathcal{H}_+:=\mathrm{int}\left(\mathcal{S}_+\cup\mathcal{W}_+\right)
\end{equation}
are two open (possibly empty) intervals of $(-M,0)$ and $(0,M)$ respectively. Therefore one has in any case
\begin{equation}
\mathcal{H}_-=(\Lambda_-,\lambda_-) \textrm{ and } \mathcal{H}_+=(\lambda_+,\Lambda_+).
\end{equation}
In the sequel, the quantities $\Lambda_-, \lambda_-$ and $\lambda_+,\Lambda_+$ are referred to as the \emph{negative and positive wave speed bounds} for $u$.
\begin{lemma}
\label{lm:hcvg}
For any given Riemann data $(u_L,u_R)$, some real $M>0$ being prescribed according to the bounds~\eqref{eq:hyp-bound}, let us consider $\{u^\eps\}_{\eps>0}$ a subsequence of solutions to \eqref{eq:Dafermos}--\eqref{eq:Boundcond} that converges as $\eps$ goes to zero to a CRD solution $u$. Then the subsequence $\{h^\eps(\cdot;u^\eps)\}_{\eps>0}$ given by \eqref{eq:heps} converges uniformly over $\xi\in(-M,M)$ to the function $h(\cdot;u)$ defined by \eqref{eq:h}.\\
In addition, the function $\xi\mapsto h(\xi;u)$ is continuous, piecewise differentiable over $(-M,M)$, convex over $(-M,0)$, convex over $(0,M)$. Setting $\Lambda_-,\lambda_-$ and $\lambda_+,\Lambda_+$ as the negative and positive wave speed bounds according to~\eqref{eq:boundwaves}, then one has the (almost everywhere) formula
\begin{equation}
\label{eq:hprime}
h'(\xi;u) = 
\begin{cases}
\xi - f'_-(u_L), & \xi\in(-M,\Lambda_-),\\
0, & \xi \in (\Lambda_-,\lambda_-),\\
\xi - f'_-(u_-), & \xi\in(\lambda_-,0),\\
\xi - f'_+(u_+), & \xi\in(0,\lambda_+),\\
0, & \xi \in (\lambda_+,\Lambda_+),\\
\xi - f'_+(u_R), & \xi\in(\Lambda_+,M).
\end{cases}
\end{equation}
\end{lemma}
\begin{proof}
 Let $\xi$ be fixed in $(-M,M)$, we obtain the following rough upper bound
\begin{align*}
 &|h^\eps(\xi;u^\eps)-h(\xi;u)|\\
   &\leq \frac12\left|\int_0^\xi f_-'(u^\eps(s))(1-v^\eps(s)) - f_-'(u(s))\indic_{s<0}\, ds \right|\\
       & + \frac12\left|\int_0^\xi f_+'(u^\eps(s))(1+v^\eps(s)) - f_+'(u(s))\indic_{s>0}\,  ds \right| \\
  & \leq \int_{-M}^M \left|f_-'(u^\eps(s))-f_-'(u(s))\right|\indic_{s<0}\, ds
    + \frac12\int_{-M}^M \left|f_-'(u^\eps(s))\right| \left|\indic_{s>0} - \indic_{s<0} - v^\eps(s)\right| ds \\
  & +\int_{-M}^M \left|f_+'(u^\eps(s))-f_+'(u(s))\right|\indic_{s>0}\, ds
    + \frac12\int_{-M}^M \left|f_+'(u^\eps(s))\right| \left|\indic_{s>0} - \indic_{s<0} - v^\eps(s)\right| ds,
\intertext{and thus}
 |h^\eps(\xi;u^\eps)-h(\xi;u)|
 \leq 
 & \left(\|f_-''\|_\infty+\|f_+''\|_\infty\right) \int_{-M}^M |u^\eps-u|ds\\
  &\phantom{=}+ \frac12\left(\|f_-'\|_\infty+\|f_+'\|_\infty\right)\int_{-M}^M \left|\indic_{s>0} - \indic_{s<0} - v^\eps(s)\right| ds,
\end{align*}
where we used the identity $\indic_{s>0} + \indic_{s<0} = \indic_{s\in\RR}$ for almost every $s$.
The sequences $\{u^\eps\}_{\eps>0}$ and $\{v^\eps\}_{\eps>0}$ moreover converge in $L^1_{\loc}(\RR)$ to $u$ and to the sign function respectively, and thus $\{h^\eps(\cdot;u^\eps)\}_{\eps>0}$ converges uniformly to $h(\cdot;u)$ over $(-M,M)$.\\
The function $u$ is piecewise continuous over $(-M,M)$. More precisely for $\xi\in(-M,\Lambda_-)$ the solution $u(\xi)$ is nothing but the left Riemann data $u_L$, while for $\xi\in(\lambda_-,0)$ (if non-empty) $u(\xi)=u_-$. Finally for $\xi\in(\Lambda_-,\lambda_-)$ the wave fan $u(\xi)$ is nothing but the entropy wave fan for the Riemann problem with flux $f_-$, left data $u_L$ and right data $u_-$, i.e. the succession of rarefaction waves for $\xi\in\mathcal{W}_-$ and shock waves for $\xi\in\mathcal{S}_-$.
The same structure is observed over the right half-space, with consistent notations.

The points of discontinuity of $u$ are located in the (at most) countable set $\{0\}\cup\mathcal{S}_-\cup\mathcal{S}_+$. Therefore the function $h(\cdot;u)$ is piecewise differentiable at any point $\xi\in\mathcal{C}_-\cup\mathcal{W}_-\cup\mathcal{W}_+\cup\mathcal{C}_+$ and one has then $h'(\xi;u)=\xi-f_-'(u(\xi))$ if $\xi<0$ and $h'(\xi;u)=\xi-f_+'(u(\xi))$ if $\xi>0$. Moreover, for $\xi\in\mathcal{W}_\pm$, it is well-known that
\begin{equation}
\label{eq:raref}
f_\pm'(u(\xi))=\xi
\end{equation}
so that the formula \eqref{eq:hprime} is satisfied almost everywhere.

It remains to prove the convex character of $h(\cdot;u)$ over $(-M,0)$ (a similar proof apply over $(0,M)$). To that aim, {let us prove that} $h'(\xi;u)$ is a non-decreasing function of $\xi\in(-M,0)$. Thanks to \eqref{eq:hprime}, this function is non-decreasing over each of the open intervals $(-M,\Lambda_-)$, $(\Lambda_-,\lambda_-)$ and $(\lambda_-,0)$. Thence it remains to put in order the left and right traces of $h'(\cdot;u)$ at $\xi=\Lambda_-$ and $\xi=\lambda_-$. Two situations may arise: either $\Lambda_-\in\mathcal{W}_-$ or $\Lambda_-\in\mathcal{S}_-$ (and similarly for $\lambda_-$). In the first case, $u$ is continuous at $\Lambda_-$ with $u(\Lambda_-)=u_L$ so that $h(\cdot;u)$ is continuous at $\xi=\Lambda_-$ with $h(\Lambda_-;u)=0$. In the second case, the Lax inequalities for the entropy discontinuity located at $\xi=\Lambda_-$ read
\begin{equation*}
f_-'(u_L)=f_-'(u(\Lambda_-^-))>\Lambda_->f_-'(u(\Lambda_-^+)).
\end{equation*}
In that case, $h'(\Lambda_-^-;u)=\Lambda_- - f'_-(u_L) < 0 = h'(\Lambda_-^+;u)$.
The same line of reasoning apply at $\lambda_-$ if non zero, with $u_-$ in place of $u_L$ and reversed inequalities.
\end{proof}

In Figure~\ref{fig:hfunction}, we illustrate the previous properties for two different solutions $u(\xi)$ (bottom plots) and the corresponding structure for the corresponding functions $h(\xi;u)$ (top plots). The two graphs on the left figure concern a solution that is continuous at $\xi=0$ while the two others ones in the right {represent} a discontinuous solution. Notice here that all rarefaction waves are for convenience figured with affine lines. Observe that the first example presents a left wave fan with a sticked rarefaction on the left ($\Lambda_-\in\mathcal{W}_-$ so that $h'(\Lambda_-;u)=0$) and a shock wave on the right ($\lambda_-\in\mathcal{S}_-$ so that $h'(\lambda_-^-;u)<h'(\lambda_-^+;u)=0$).
According to Lemma~\ref{lm:hcvg}, the functions $h(\cdot;u)$ are piecewise convex with constant plateau over open intervals in which the solution $u$ presents a wave fan. These figures also illustrate Proposition~\ref{prop:MR3}, whose proof follows. The sets $\supp\,u'=\argmin_\xi h(\xi;u)$ are represented with thick line and dots, along the horizontal axis on the $u$-plots, and along the ordinate $\min h$ on the $h(\xi,u)$-plots.

Due to the third point in Corollary~\ref{cor:MR3}, the solution represented on the right, that is discontinuous at the interface has to satisfy the following property: $f_-'(u_-)=f_+'(u_+)=0$. In other words, any trace of the solution at the interface corresponds to a sonic point for the corresponding flux function.

\begin{figure}
\begin{center}
\includegraphics[scale=0.90]{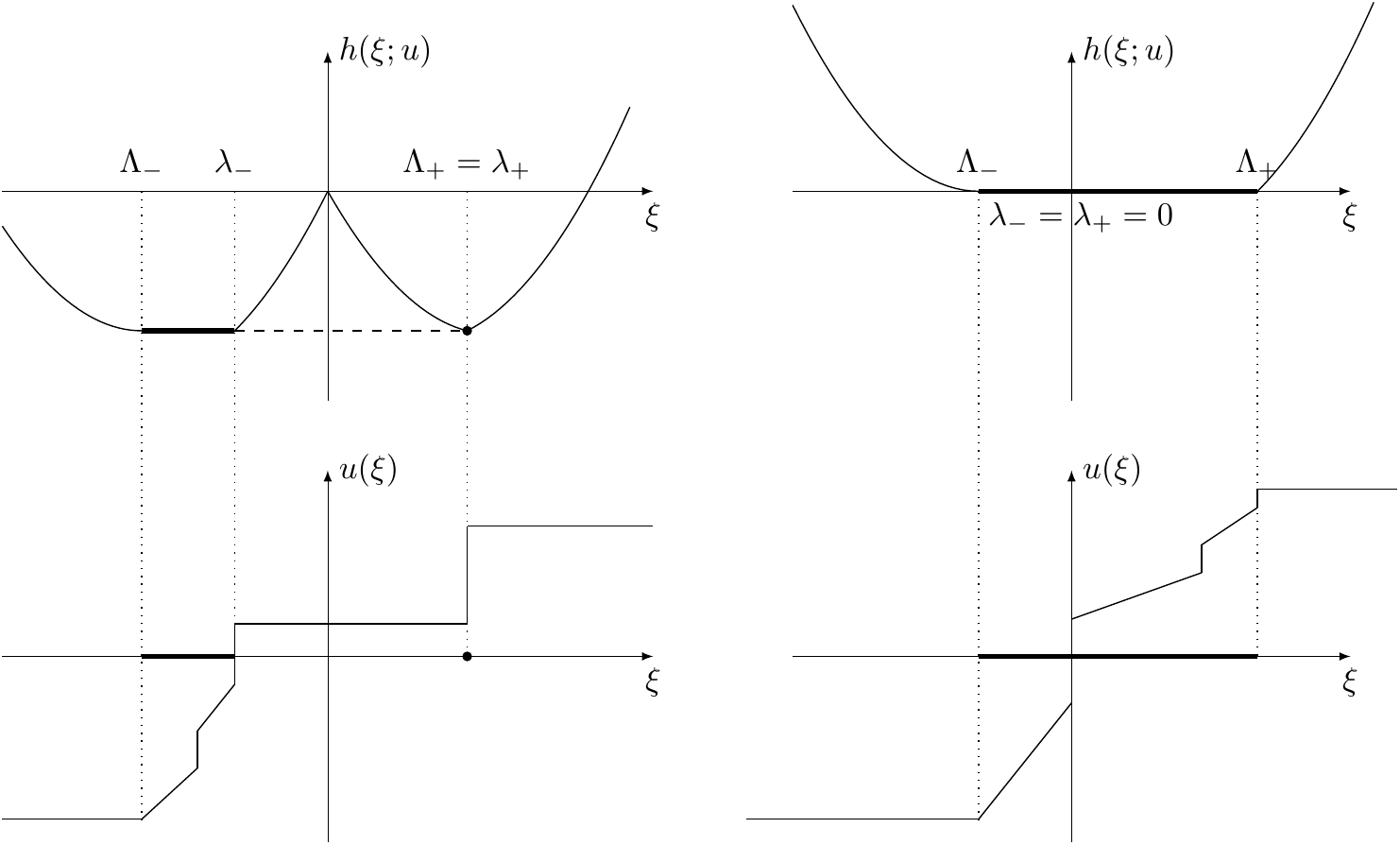}
\end{center}
\caption{Illustration of case (\ref{cor:main3point2}) of the Corollary~\ref{cor:MR3}. The solution $u$ is continuous at $\xi=0$ (left) vs. discontinuous at $\xi=0$ (right)}
\label{fig:hfunction}
\end{figure}

Now we are ready to prove Proposition~\ref{prop:MR3}.

\begin{proof}[Proof of Proposition~\ref{prop:MR3}]
From Lemma~\ref{lm:hcvg}, we get the inclusion $\argmin_{\xi} h(\xi;u)\subset[\Lambda_-,\lambda_-]\cup[\lambda_+,\Lambda_+]$ and thus immediately:
\begin{equation}
\argmin_{\xi} h(\xi;u) \subset \supp\ u'.
\end{equation}
Let us now consider some $\xi \notin \argmin_\xi h(\xi;u)$ and prove that necessarily $\xi \notin \supp\ u'$. Let us introduce the quantity $m=\min h(\cdot;u)$, the minimum being taken over $[-M,M]$, and  by $\xi_m\in[-M,M]$ a point such that $h(\xi_m;u) = m$. The uniform convergence of the sequence $\{h^\eps(\cdot;u^\epsilon)\}_{\eps>0}$ to $h(\cdot;u)$ on the one hand and the continuity of $h(\cdot;u)$ on the other hand ensure that for $\delta>0$ sufficiently small being given, there exists $\epsilon_0>0$ and $\eta>0$ such that
\begin{equation*}
\forall \epsilon<\epsilon_0,\ \forall s\in [-M,M],\ |s-\xi|<\eta\ \Rightarrow h^\eps(s;u^\epsilon) \geq m + \delta.
\end{equation*}
From \eqref{eq:repres} we then get, for any $\zeta\in[-M,M],\ |\zeta-\xi|<\eta$,
\begin{equation}
\label{eq:represbis}
u^\eps(\xi)-u^\eps(\zeta) = (u_R-u_L) \dfrac{\displaystyle\int_\zeta^\xi e^{-h^\eps(s;u^\eps)/\eps}\, ds}{\displaystyle\int_{-M}^M e^{-h^\eps(s;u^\eps)/\eps}\, ds},
\end{equation}
where
\begin{equation*}
\left|\int_\zeta^\xi e^{-h^\eps(s;u^\eps)/\eps}\, ds\right | \leq |\xi-\zeta| e^{-(m+\delta)/\epsilon}.
\end{equation*}
Again, thanks to the uniform convergence of $\{h^\eps(\cdot;u^\epsilon)\}_{\eps>0}$ and the continuity of $h(\cdot;u)$, there exists $\tilde\eta>0$ such that
\begin{equation*}
\forall \epsilon<\epsilon_0,\ \forall s\in[\xi_m-\tilde\eta,\xi_m+\tilde\eta],\ h^\eps(s;u^\eps)\leq m+\delta/2.
\end{equation*}
Therefore the lower integral in \eqref{eq:represbis} is bounded from below as follows:
\begin{equation*}
\left|\int_{-M}^M e^{-h^\eps(s;u^\eps)/\eps}\, ds\right| \geq 2\tilde\eta e^{-(m+\delta/2)/\epsilon}.
\end{equation*}
Finally, we get from \eqref{eq:represbis} and the above inequalities
\begin{equation*}
|u^\eps(\xi)-u^\eps(\zeta)| \leq \left|\dfrac{(u_R-u_L)(\xi-\zeta)}{2\tilde\eta}\right| e^{-\delta/(2\epsilon)}.
\end{equation*}
This last quantity vanishes as $\eps$ goes to zero and thus $u(\zeta)=u(\xi)$. Therefore the limiting solution $u$ is constant over $]\xi-\eta,\xi+\eta[$ and $\xi\notin \supp\ u'$.
\end{proof}

We continue and finish this section with the proof of Corollary~\ref{cor:MR3}.

\begin{proof}[Proof of Corollary~\ref{cor:MR3}]
\ 
\begin{enumerate}
\renewcommand{\theenumi}{\roman{enumi}}
\renewcommand{\labelenumi}{(\theenumi)}
\item The proof relies on a contraposition argument. Suppose that $\supp\,u'\cap(0,M)\neq \emptyset$ then, thanks to Proposition~\ref{prop:MR3}, one has $\argmin_{\xi\in[-M,M]} h(\xi;u)\cap(0,M)\neq\emptyset$. In other words, $h(\cdot;u)$ reaches its global minimum at some $\xi_m>0$.\\
Observe then $h(0;u)=0$ so that necessarily $h(\xi_m;u)\leq h(0;u) \leq 0$ and due to the convex character of $h(\cdot;u)$ over $(0,M)$ (from Lemma~\ref{lm:hcvg}), $h(\xi;u)\leq 0$ for all $\xi\in[0,\xi_m]$. Consequently, consider $\xi\in(0,\xi_m)$ then
\begin{equation*}
\frac{1}{\xi}\int_0^\xi (s-f_+'(u(s)))\,ds \leq 0, 
\end{equation*}
so that
\begin{equation*}
\frac{1}{\xi}\int_0^\xi f_+'(u(s))\,ds \geq \frac{1}{\xi}\int_0^\xi s\, ds = \frac \xi2 \geq 0, 
\end{equation*}
and, passing to the limit $\xi$ tends to zero, one gets
$
f_+'(u_+)\geq 0.
$
The same line of reasoning apply for the second part of the proposition.

\item We just proved that if $I_-:=\supp\,u'\cap(-M,0)\neq \emptyset$ and $I_+:=\supp\,u'\cap(0,M)\neq \emptyset$ then both inequalities are satisfied: $f_-'(u_-)\leq 0$ and $f_+'(u_+)\geq 0$.\\
Moreover, Proposition~\ref{prop:MR3} ensures that $h(I_-;u)=h(I_+;u)=\{\min_{\xi\in[-M,M]}h(\xi;u)\}$ and the conclusion follows.

\item  We already understood in point (i) that assuming $\supp\,u'\cap(0,M)\neq\emptyset$ one has $f_+(u_+)\geq 0$.
\item 
Assume now the discontinuity at $\xi=0$: $u_-\neq(0^+)$, then
\[0\in \supp\, u' = \argmin_{\xi\in[-M,M]}h(\xi;u),\]
and thus $\min h(\cdot;u)=h(0;u)=0$.
Let $\xi_m>0$ be some point in $\supp\,u'\cap(0,M)$, with therefore $h(\xi_m;u)=0$, and following the same argument as in (i) based on both the convex character and the positiveness of $h(\cdot;u)$ one gets $\forall \xi\in[0,\xi_m]\ h(\xi;u)=0$, so that
\begin{equation*}
\frac{1}{\xi}\int_0^\xi f_+'(u(s))\,ds = \frac{1}{\xi}\int_0^\xi s\, ds = \frac \xi2.
\end{equation*}
Passing to the limit $\xi$ tends to zero gives the expected result: $f_+'(u_+)=0$.
Of course, the same line of reasoning apply for the second part of the proposition.
\end{enumerate}
\end{proof}

%
%
%

\section{Analysis of the interfacial layer}
\label{sec:layer}

In this section, we characterize the behavior of the limiting solution across the interface $\{\xi=0\}$ revealing the possible boundary layers connecting the traces $u_-$ and $u_+$ of $u$ at $\xi=0^{-}$ and at $\xi=0^{+}$ respectively.
To that purpose, we proceed to a classical \emph{blow-up} at the origin through the change of variable $\xi=\eps y$. Let us define
\begin{equation}
\label{eq:blowup}
\calU^\eps(y)=u^\eps(\eps y),\qquad
\calV^\eps(y)=v^\eps(\eps y),\qquad \textrm{for }\ \eps|y|\leq M.
\end{equation}
For convenience the functions $\calU^\eps$ and $\calV^\eps$ are defined over the whole real line, using the natural boundary conditions data~\eqref{eq:Boundcond} as constant extensions outside the above domain. The results of this section are summarized in the following Theorem.
\begin{theorem}
\label{thm:RelaxLayer}
Let $f_-$ and $f_+$ be two functions in $\calC^1(\RR)$ such that $f_\pm'$ are Lipschitz continuous, let $u^\eps\in L^\infty(\RR)\cap BV(\RR)$ and $v^\eps\in L^\infty(\RR)\cap BV(\RR)$ be solutions of \eqref{eq:Dafermos}--\eqref{eq:Boundcond}, and {consider} $\calU^\eps$ and $\calV^\eps$ given by \eqref{eq:blowup}. Then the following facts are satisfied.
\begin{enumerate}

 \item \label{thm:RL1} The sequence $\{\calV^\eps\}_{\eps>0}$ converges uniformly as $\eps>0$ to $\calV\in\calC^\infty(\RR)$ a bounded monotone increasing function, given by
 \begin{equation}
 \label{eq:limitV}
 \calV(y):=-1+2 \int_{-\infty}^y e^{-{s^2/2}}\,ds\left/\int_{-\infty}^{+\infty} e^{-{s^2/2}}\,ds\right..
\end{equation}
More precisely, there exists a positive constant $C$, independent of $\eps>0$, such that the following estimate holds:
\begin{equation}
\label{eq:LinfestimV}
\|\calV^\epsilon-\calV\|_{L^\infty (\RR)}\leq C e^{-M^2/2\epsilon^2}.
\end{equation}

 \item\label{thm:RL2} Up to a subsequence, $\{\calU^\eps\}_{\eps>0}$ converges strongly in $L^1_{\rm loc}(\RR)$ to $\calU\in\calC^2(\RR)$, a bounded and monotone function, whose monotonicity corresponds to the sign of $u_R-u_L$. The limit $\calU$ is a solution to the viscous profile ODE:
\begin{equation}
\label{eq:RLeqU}
 \dfrac12\left((1-\calV)f_-'(\calU)+(1+\calV)f_+'(\calU)\right)\calU' = \calU'',
\end{equation}
and admits limits at infinities, denoted by $\calU_{-\infty}$ and $\calU_{+\infty}$ respectively, and that satisfy the inequalities
\begin{equation}
\label{eq:Linfbound}
 \min(u_L,u_R) \leq \calU_{\pm\infty} \leq \max(u_L,u_R).
\end{equation}

\item\label{thm:RL3} The inner profile is non trivial i.e. $\calU'(y)\neq 0$ for all finite $y$ in $\RR$ (then necessarily $\calU'(y)$ vanishes as $|y|$ goes to infinity) if and only if the following two asymptotic conditions are fulfilled:
\begin{gather}
\lim_{y\to+\infty} \int_0^y f'_+(\calU(s))\, ds = -\infty,\label{eq:condasymptR}\\
\lim_{y\to-\infty} \int_y^0 f'_-(\calU(s))\, ds = +\infty.\label{eq:condasymptL}
\end{gather}

\end{enumerate}

\end{theorem}

\bigskip

As we will prove, the asymptotic conditions \eqref{eq:condasymptR}--\eqref{eq:condasymptL} reflects the property that the inner profile must stay uniformly bounded in $y$. In particular, if one of the following two conditions $f'_-(\calU_{-\infty})<0$ or $f'_+(\calU_{+\infty})>0$ is satisfied, then the inner solution must stay constant.

By contrast, both the conditions $f'_-(\calU_{-\infty})>0$ and $f'_+(\calU_{+\infty})<0$ clearly suffice to imply \eqref{eq:condasymptR}--\eqref{eq:condasymptL} and in fact these conditions reflect that the orbit $\calU(y)$ connects transversally the unstable endpoint $\calU_{-\infty}$ to the stable endpoint $\calU_{+\infty}$, in the sense of the {Hartman-Grobman theory \cite{Shivamoggi:2014}}.
Observe that \eqref{eq:condasymptR}--\eqref{eq:condasymptL} may be valid with vanishing wave velocities at one or even at both endpoints. Some examples of such situations will be proposed hereafter.

\begin{proof}[Proof of point (\ref{thm:RL1}) of Theorem~\ref{thm:RelaxLayer}]
{Let us first consider} some given $y\in(-M/\eps,M/\eps)$ and compute then
\begin{multline*}
|\calV^\eps(y)-\calV(y)|
	= 2\left| \dfrac{\int_{-M/\epsilon}^{y} e^{-{s^2/2}}\,ds}{\int_{-M/\eps}^{+M/\eps} e^{-{s^2/2}}\,ds} 
		- \dfrac{\int_{-\infty}^y e^{-{s^2/2}}\,ds}{\int_{-\infty}^{+\infty} e^{-{s^2/2}}\,ds} \right|\\
	= 2\dfrac{\left|\int_{-\infty}^{+\infty} e^{-{s^2/2}}\,ds\int_{-M/\eps}^{y} e^{-{s^2/2}}\,ds
		- \int_{-M/\eps}^{+M/\eps} e^{-{s^2/2}}\,ds\int_{-\infty}^{y} e^{-{s^2/2}}\,ds \right|}{\int_{-M/\eps}^{+M/\eps} e^{-{s^2/2}}\,ds\int_{-\infty}^{+\infty} e^{-{s^2/2}}\,ds}.
\end{multline*}
The denominator reads $(\int_{-\infty}^{+\infty} e^{-{s^2/2}}\,ds)^2 + o(\eps) = 2\pi + o(\eps)$. Consequently there exists a positive constant $c$ such that
\begin{align*}
|\calV^\eps(y)-\calV(y)| &\leq c \left|\int_{-\infty}^{+\infty} e^{-{s^2/2}}\,ds\int_{-M/\eps}^{y} e^{-{s^2/2}}\,ds
	- \int_{-M/\eps}^{+M/\eps} e^{-{s^2/2}}\,ds\int_{-\infty}^{y} e^{-{s^2/2}}\,ds \right|\\
	&\leq c \left|\int_{-\infty}^{+\infty} e^{-{s^2/2}}\,ds - \int_{-M/\eps}^{+M/\eps} e^{-{s^2/2}}\,ds\right|\ \left|\int_{-M/\eps}^{y} e^{-{s^2/2}}\,ds \right|\\
		&\hspace{1em}+ c \left|\int_{-M/\eps}^{y} e^{-{s^2/2}}\,ds - \int_{-\infty}^{y} e^{-{s^2/2}}\,ds\right|\ \left|\int_{-M/\eps}^{M/\eps} e^{-{s^2/2}}\,ds \right|\\
	&\leq 3 c \left|\int_{-\infty}^{+\infty} e^{-{s^2/2}}\,ds \right|\ \left|\int_{M/\eps}^{+\infty} e^{-{s^2/2}}\,ds \right|\\
	&\leq C e^{-M^2/2\eps^2},
\end{align*}
due to the the asymptotic behavior of the usual error function.\\
Considering now some $y\in\RR$ with $|y|\geq M/\eps$, the same last argument gives
\begin{equation*}
|\calV^\eps(y)-\calV(y)|
	= |1-\calV(y)| \leq C e^{-y^2/2} \leq C e^{-L^2/2\eps^2}.\qedhere
\end{equation*}
\end{proof}

\bigskip

\begin{proof}[Proof of point~(\ref{thm:RL2}) of Theorem~\ref{thm:RelaxLayer}]
The rescaled profile $\calU^\eps$ clearly achieves the same monotonicity property as $u^\eps$ and stays uniformly bounded in sup-norm since $u^\eps$ does. As a consequence, $\calU^\eps$ is uniformly bounded in $BV(\RR)$.\\
Thanks to Helly's Theorem, a classical diagonal extraction procedure yields the existence of an extracted subsequence, still denoted $\{\calU^\eps\}_{\eps>0}$, which strongly converges in $L^1_{\rm loc}(\RR)$ to some limit profile $\calU\in L^\infty(\RR)\cap BV(\RR)$, with the following property:
\begin{equation}
\min(u_L,u_R) \leq \calU(y) \leq \max(u_L,u_R),\quad y\in\RR.
\end{equation}
Due to its monotonicity property, $\calU$ admits finite limits as $y$ goes to $+\infty$ (respectively $-\infty$), {which} we denote $\calU_{+\infty}$ (resp. $\calU_{-\infty}$), that also satisfy \eqref{eq:Linfbound}.

Let us now prove that $\calU$ actually solves \eqref{eq:RLeqU}. We will first show that $\calU$ satisfies \eqref{eq:RLeqU} in the usual sense of the distributions {and} then observe that $\calU$ is indeed a strong solution due to the smoothness of $\calV(y)$.

Let $\eps>0$ be fixed. {Let $\xi\in(-M,M)$ be given} so that the first equation in \eqref{eq:Dafermos} rewrites in the fast variable $y\in(-M/\eps,M/\eps)$
\begin{equation*}
 -\eps y \calU^\eps{}'(y) + \dfrac12\left((1-\calV^\eps(y))f_-'(\calU^\eps(y))+(1+\calV^\eps(y))f_+'(\calU^\eps(y))\right)\calU^\eps{}'(y) = \calU^\eps{}''(y).
\end{equation*}
Let $\varphi\in\calC_c^\infty(\RR)$ be a test function and, for $\eps >0$, let $\chi^\eps$ be a standard smooth cut off function with compact support included in $]-M/\eps,M/\eps[$ and such that $\|\chi^\eps\|_{L^\infty(\RR)}=1$. {Let us now introduce} $\phi^\eps=\varphi \chi^\eps$ so that $\phi^\eps\in\calC^\infty_c(]-M/\eps,M/\eps[)$. Using classical integration by parts, the weak form of the above equation reads:
\begin{equation*}
\dfrac12\int_\RR f_+(\calU^\eps)\{\phi^\eps{}' + \phi^\eps \calV^\eps{}'\} + f_-(\calU^\eps)\{\phi^\eps{}' - \phi^\eps \calV^\eps{}'\} \, dy + \int_\RR \calU^\eps \phi^\eps{}'' \,dy = -\eps \int_\RR \calU^\eps{}' (y\phi^\eps) \,dy.
\end{equation*}
We thus infer
\begin{align*}
&\left| \dfrac12 \int_\RR f_+(\calU^\eps)\{\phi^\eps{}' + \phi^\eps \calV{}'\} + f_-(\calU^\eps)\{\phi^\eps{}' - \phi^\eps \calV{}'\} \, dy + \int_\RR \calU^\eps \phi^\eps{}'' \,dy \right| \\
& \leq \|f_+(\calU^\eps)-f_-(\calU^\eps)\|_{L^\infty(\RR)}\ \|\varphi\|_{L^1(\RR)}\ \|\calV-\calV^\eps\|_{L^\infty(\RR)} + \eps\ TV_\RR(\calU^\eps)\ \|(y\varphi)\|_{L^\infty(\RR)}\\
& \leq C \eps,
\end{align*}
for some $C>0$ independent of $\eps$, thanks to \eqref{eq:LinfestimV} and since $\calU^\eps$ is uniformly bounded in $BV$-norm. Then observe that $|f_\pm(\calU^\eps)|\leq C |\calU^\eps| + |f_{\pm}(0)|$ with $f_\pm(\calU^\eps(y))\to f_\pm(\calU(y))$ almost everywhere for $y\in\RR$. The Lebesgue dominated convergence theorem applies to get in the limit $\eps$ goes to 0:
\begin{equation}
\label{eq:weakform}
\dfrac12\int_\RR f_+(\calU)\{\varphi' + \varphi \calV'\} + f_-(\calU)\{\varphi' - \varphi \calV'\} \, dy + \int_\RR \calU\, \varphi'' \,dy = 0,
\end{equation}
for all test function $\varphi\in\calC^\infty_c(\RR)$.\\
Hence the limit profile $\calU$ solves the weak form of \eqref{eq:RLeqU}. Observe that \eqref{eq:RLeqU} is a second order differential equation with a smooth varying coefficient $\calV(y)$ so that $\calU$ in \eqref{eq:weakform} also solves \eqref{eq:RLeqU} in the strong sense. This solution is clearly defined for all $y\in\RR$ and is at least twice differentiable in view of the smoothness assumption on the functions $f_-$, $f_+$, and $\calV$.
\end{proof}

\bigskip
\begin{proof}[Proof of point~(\ref{thm:RL3}) of Theorem~\ref{thm:RelaxLayer}]
Let us prove at last the asymptotic conditions \eqref{eq:condasymptR}-\eqref{eq:condasymptL} in the case of a nonconstant inner solution $\calU$. We may thus assume that $\calU'(0)\neq 0$.
We prove that $\calU'(y)$ goes to zero as $|y|$ goes to infinity. A direct integration of the governing equation \eqref{eq:RLeqU} yields for all $y>0$:
\[\calU'(y) = \calU'(0)\exp\left(\int_0^y f_+'(\calU(s))ds\right)\exp\left(\int_0^y \frac{1-\calV(s)}{2}\left(f_-'(\calU(s))-f_+'(\calU(s))\right)ds\right).\]
In one hand, observe that $\calV(s)$ goes exponentially fast to 1 as $s$ goes to infinity, while $f_-'(\calU(s))-f_+'(\calU(s))$ stays uniformly bounded. Hence the second exponential factor converges to a strictly positive limit $\ell>0$ as $y$ goes to infinity.\\
In the other hand, let us define now
\[\calL(y):=\exp\left(\int_0^y f'_+(\calU(s))\,ds\right),\]
which is clearly positive for all finite $y>0$.
Next, $\calU(s)$ monotonically reaches a finite limit $\calU_{+\infty}$ as $s$ goes to infinity, but since $f_+$ admits a finite number of inflection points, necessarily $f'_+(\calU(s))$ keeps a constant sign for large enough values of $s$. Consequently $\calL(y)$ admits a nonnegative limit as $y$ tends to $+\infty$, which may be finite or not. Assume that this limit is strictly positive, i.e. assume that there exists a strictly positive $\omega>0$ such that,
\begin{equation}\label{hyp:contrad}\calL(y)>\omega>0,\end{equation}
for all $y>0$, then necessarily $(u_R-u_L)\calU'(y)>(u_R-u_L)\calU'(0)\omega \ell$ since $(u_R-u_L)\calU'(0)>0$, so that by integration $(u_R-u_L)(\calU(y)-\calU(0))>(u_R-u_L)\calU'(0)\omega \ell y$.
This rises a contradiction with the uniform boundedness of the inner solution $\calU(y)$; so that \eqref{hyp:contrad} cannot hold true. Hence necessarily $\calL(y)$ must vanish as $y$ goes to infinity. Necessarily $\int_0^y f'_+(\calU(s))\, ds$ must tend to $-\infty$ in this limit.\\
The companion asymptotic condition \eqref{eq:condasymptL} can be proved following the same steps.
\end{proof}


\section{Study of the matching conditions}
\label{sec:matching}

In the present section, we study the matching conditions in between the different interfacial quantities involved in the previous statements: the traces nearby the coupling interface $u_-$ and $u_+$, and the limits of the internal coupling layer $\calU_{-\infty}$ and $\calU_{+\infty}$. The following Theorem summarizes the results of this section.

\begin{theorem}
\label{thm:matchCond}
Under the assumptions of Theorem~\ref{thm:RelaxLayer} and with the same notations, we have the following matching conditions in between the traces $u_{\pm}$ of $u$ at $\xi=\pm 0$ and the limits $\calU_{\pm\infty}$ of $\calU$ at $y=\pm\infty$.
\begin{itemize}
\item The right trace $u_+$ and the endpoint $\calU_{+\infty}$ satisfy 
\begin{equation}
\label{eq:monotL}
(u_R-u_L)(u_+-\calU_{+\infty})\geq 0
\end{equation} with
\begin{equation}
f_+(\calU_{+\infty}) = f_+(u_+),
\end{equation}
 together with
\begin{gather}
\sgn(u_+-k)(f_+(u_+)-f_+(k))\leq 0,\label{eq:OleinikL1}\\
\sgn(\calU_{+\infty}-k)(f_+(\calU_{+\infty})-f_+(k))\geq 0,\label{eq:OleinikL2}
\end{gather}
for all $k\in [\calU_{+\infty},u_+]$. {In particular, as soon as} $\calU_{+\infty}\neq u_+$,
\begin{equation}
\label{eq:RLconsR}
f'_+(\calU_{+\infty})\geq 0 \textrm{ and } f'_+(u_+)\leq 0,
\end{equation}
with strict inequalities in the case of a genuinely non linear flux $f_+$.

\smallskip
\item The left trace $u_-$ and the endpoint $\calU_{-\infty}$ satisfy 
\begin{equation}
\label{eq:monotR}
(u_R-u_L)(\calU_{-\infty}-u_-)\geq 0
\end{equation}
with
\begin{equation}
 f_-(u_-)= f_-(\calU_{-\infty}),
\end{equation} 
  together with
\begin{gather}
\sgn(u_--k)(f_-(u_-)-f_-(k))\geq 0,\label{eq:OleinikR1}\\
\sgn(\calU_{-\infty}-k)(f_-(\calU_{-\infty})-f_-(k))\leq 0,\label{eq:OleinikR2}
\end{gather}
for all $k\in [u_-,\calU_{-\infty}]$. In particular, and as soon as $\calU_{-\infty}\neq u_-$,
\begin{equation}
\label{eq:RLconsL}
f'_-(u_-)\geq 0\textrm{ and } f'_-(\calU_{-\infty})\leq 0 ,
\end{equation}
with strict inequalities in the case of a genuinely non linear flux $f_-$.

\smallskip
\item In addition and in the case of a non trivial inner solution, the endpoints $\calU_{-\infty}\neq \calU_{+\infty}$ must obey
\begin{equation}
(u_R-u_L)(\calU_{+\infty}-\calU_{-\infty})>0,
\end{equation}
and
\begin{equation}
\label{eq:nontrivial}
f'_-(\calU_{-\infty})\geq 0 \textrm{ and } f'_+(\calU_{+\infty})\leq 0.
\end{equation}

\end{itemize}
\end{theorem}

Let us rephrase the above statement as follows. If the endpoint $\calU_{+\infty}$ differs from $u_+$, then a standing entropy shock wave for the right flux $f_+$ sticks along the interface. A similar situation occurs for non-matching values of $\calU_{-\infty}$ and $u_-$ but for the left flux $f_-$.
In addition, these two standing shocks may coexist as well with a non trivial inner layer $\calU$ connecting two distinct endpoints $\calU_{-\infty}$ and $\calU_{+\infty}$ but then with vanishing wave velocities $f'_-(\calU_{-\infty})=f'_+(\calU_{+\infty})=0$. Let us indeed give an example of such an exotic situation.
As asserted by Theorem~\ref{thm:matchCond}, this situation only takes place in the case of a pair of fluxes with no convexity property.
In that aim, we choose double well flux pairs as depicted in Figure~\ref{fig:doublewell}. Here $u^\star_{1-}<u^\star_{2-}$ (respectively $u^\star_{1+}<u^\star_{2+}$) denote the two sonic points of $f_-$ (respectively $f_+$) with the property that $f_+(u^\star_{1+})=f_+(u^\star_{2+})$ and $f_-(u^\star_{1-})=f_-(u^\star_{2-})$.

\begin{figure}[!ht]
\begin{center}
\begin{tikzpicture}[>=latex,scale=0.9]
 \draw[->] (-3,0) -- (3,0) node[below right] {$u$};
 \draw[->] (-3,0) -- (-3,4) node[above left] {$f(u)$};
 \begin{scope}
 \draw plot[domain=-0.15:2.15,smooth] (\x,{1.25+2*(\x-1)*(\x-1)*(\x-1)*(\x-1) - 0.8*(\x-1)*(\x-1)}) node[right] {$f_+$} ;
 \draw plot[domain=-2.5:0.5,smooth] (\x,{0.75+0.9*(\x+1)*(\x+1)*(\x+1)*(\x+1) - 1*(\x+1)*(\x+1)}) node[right] {$f_-$} ;
 \end{scope}
 \begin{scope}
 \draw (-1.7453,0.45) -- (-0.2546,0.45);
 \draw (0.5527,1.15) -- (1.4472,1.15);
 \end{scope}
 \begin{scope}[dotted]
 \draw (-0.2546,0.45) -- (-0.2546,0) node[below] {$u^\star_{2-}$};
 \draw (-1.7453,0.45) -- (-1.7453,0) node[below] {$u^\star_{1-}$};
 \draw (0.5527,1.15) -- (0.5527,0) node[below] {$u^\star_{1+}$};
 \draw (1.4472,1.15) -- (1.4472,0) node[below] {$u^\star_{2+}$};
 \end{scope}
\end{tikzpicture}%
\end{center}%
\caption{Standing shocks and non trivial inner solution}
\label{fig:doublewell}
\end{figure}
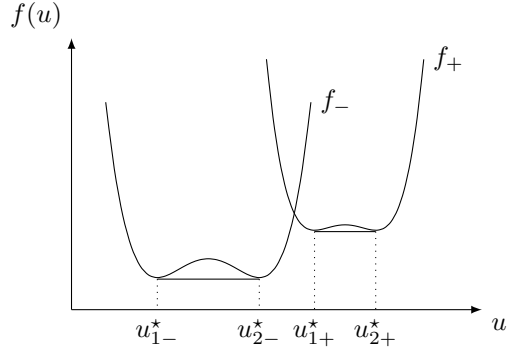
Let us prescribe the initial data as follows $u_L=u^\star_{1-}$ and $u_R=u^\star_{2+}$. The CRD solution is known to exist and monotonically increases from $u^\star_{1-}$ with $f'_-(u)>0$, for $u>u^\star_{1-}$ so that necessarily $u(\xi)=u^\star_{1-}$ for all $\xi<0$. Similar arguments allow to prove $u(\xi)=u^\star_{2+}$ for all $\xi>0$.  In other words, the coupled self-similar solution coincides with the initial data.

The left Riemann data $u_L$ being fixed, a continuation argument in the right state $u_R$  from $u^\star_2$ to $u^\star_1$ strongly supports that the coupled solution under consideration involves a non trivial inner layer in between  two standing shocks sticked at the interface. Without further analysis of the inner solution, it is impossible to discard other structures involving a single standing shock at the coupling interface. We refer the reader to the work by Boutin {\it et al.} \cite{Boutin:2010fv} where other situations may illustrate Theorem~\ref{thm:matchCond} in the case of general pair of fluxes.

The proof of the Oleinik inequalities \eqref{eq:OleinikL1} and \eqref{eq:OleinikL2} satisfied by the right trace $u_+$ of the outer solution and the exit point $\calU_{+\infty}$ of the relaxation layer (as well as the inequalities \eqref{eq:OleinikR1} and \eqref{eq:OleinikR2} satisfied by $u_-$ and $\calU_{-\infty}$) are a consequence of the following statement.

\begin{proposition}
\label{prop:convex}
Consider any entropy-entropy flux with the notations of Proposition~\ref{prop:MR2}. Then the following entropy-like inequalities are met:
\begin{gather}
 q_+(u_+)\leq q_+(\calU_{+\infty}),\label{interf-steadyCLR}\\
q_-(\calU_{-\infty})\leq q_-(u_-).\label{interf-steadyCLL}
\end{gather}
\end{proposition}

The matching conditions $f_+(u_+)=f_-(u_-)$ expressed in Theorem~\ref{thm:matchCond} directly follows from the inequality \eqref{interf-steadyCLR}, after choosing $\eta(u)=u$ and $\eta(u)=-u$. Then it is fairly well known that the inequalities \eqref{interf-steadyCLR} set for any smooth convex entropy $\eta$ is equivalent to verify all the Kru{\v z}kov entropy inequalities:
\[
\sgn(u_+-k)(f_+(u_+)-f_+(k)) \leq \sgn(\calU_{+\infty}-k)(f_+(\calU_{+\infty})-f_+(k))
\]
for all $k$ in $\RR$ (see for instance the monograph by Godlewski and Raviart~\cite{Godlewski:1991cr}). The analysis proposed by Oleinik then applies to derive the conditions \eqref{eq:OleinikR1} and \eqref{eq:OleinikR2} (see again~\cite{Godlewski:1991cr} for the details). The corresponding matching condition for $u_-$ and $\calU_{-\infty}$ are inferred following similar steps.

The proof of Proposition~\ref{prop:convex} is based on the following two technical Lemmas:

\begin{lemma}
\label{lemma:flux}
Under the assumptions of Theorem~\ref{thm:RelaxLayer}, $\calU$ satisfies the following flux relation on each half-line.
For $y>0$,
\begin{multline}
\label{lm:fluxR}
f_+(u_+) = \frac12 \Big\{(1-\calV) f_-(\calU) + (1+\calV) f_+(\calU)\Big\}(y)\\ + \frac12 \int_y^{+\infty}\big(f_+(\calU(s))-f_-(\calU(s))\big)\calV(s)\, ds
 - \calU'(y).
\end{multline}
Moreover for $y<0$ one has 
\begin{multline}
\label{lm:fluxL}
f_-(u_-) = \frac12 \Big\{(1-\calV) f_-(\calU) + (1+\calV) f_+(\calU)\Big\}(y)
\\ - \frac12 \int_{-\infty}^y\big(f_+(\calU(s))-f_-(\calU(s))\big)\calV(s)\, ds
 - \calU'(y).
 \end{multline}
\end{lemma}

\begin{lemma}
\label{lemma:entropy}
Under the assumptions of Theorem~\ref{thm:RelaxLayer}, $\eta\in\calC^1(\RR)$ being a strictly convex entropy function associated to the entropy fluxes $q_-$ (respectively $q_+$) for the flux $f_-$ (resp. $f_-$), namely $q_+'(u) = \eta'(u) f_+'(u)$ (resp. $q_-'(u) = \eta'(u) f_-'(u)$), then $\calU$ satisfies the following entropy inequalities on each half-line.
For $y>0$,
\begin{multline}
\label{lm:entropR}
q_+(u_+) \leq \frac12 \Big\{(1-\calV) q_-(\calU) + (1+\calV) q_+(\calU)\Big\}(y)\\ + \frac12 \int_y^{+\infty}\big(q_+(\calU(s))-q_-(\calU(s))\big)\calV(s)\, ds
 - d_y(\eta(\calU))(y).
\end{multline}
Moreover for $y<0$,
\begin{multline}
\label{lm:entropL}
\frac12 \Big\{(1-\calV) q_-(\calU) + (1+\calV) q_+(\calU)\Big\}(y) - \frac12 \int_{-\infty}^y\big(q_+(\calU(s))-q_-(\calU(s))\big)\calV(s)\, ds\\
 - d_y(\eta(\calU))(y) \leq  q_-(u_-).
\end{multline}

\end{lemma}

\begin{proof}[Proof of Lemma~\ref{lemma:flux}]
The identities \eqref{lm:fluxR} and \eqref{lm:fluxL} directly come from the corresponding entropy inequalities stated in Lemma~\ref{lemma:entropy}, \eqref{lm:entropR} and \eqref{lm:entropL} respectively, by choosing $\eta(u) = u$ and $\eta(u)=-u$. Therefore we now intend to prove only Lemma~\ref{lemma:entropy}.
\end{proof}

\begin{proof}[Proof of Lemma~\ref{lemma:entropy}]
Let $y$ be fixed and $\eps>0$ be fixed, small enough such that $y\in(0,L/\eps)$. Then let $a\in\RR_+$ be any given positive real number satisfying
\begin{equation}
\label{hyp:ordering}
0<\eps y \leq a.
\end{equation}
Let us start from the following identity:
\begin{align*}
\eps (\eta(u^\eps))_{\xi\xi} &= \eps \eta'(u^\eps) u^\eps_{\xi\xi} + \eps \eta''(u^\eps) (u^\eps_\xi)^2\\
&= -\xi \eta(u^\eps)_\xi + \tfrac12( (1-v^\eps) q_-(u^\eps)_y + (1-v^\eps) q_+(u^\eps)_y) + \eps \eta''(u^\eps) (u^\eps_\xi)^2,
\end{align*}
where we have plugged the first governing equation in \eqref{eq:Dafermos} and used the definition of the entropy fluxes $q'_\pm(u) = \eta'(u) f'_\pm(u)$. Let us integrate this formula for $\xi\in(\eps y,a)$ to get, once multiplied by~$-1$:
\begin{align*}
&\eps \eta(u^\eps)_\xi(\eps y) - \eps \eta(u^\eps)_\xi(a)\\
	&\hspace{1em}- \tfrac12\{(1-v^\eps)q_-(u^\eps) + (1+v^\eps)q_+(u^\eps)\}(\eps y)
	+ \tfrac12\{(1-v^\eps)q_-(u^\eps) + (1+v^\eps)q_+(u^\eps)\}(a)\\
	&\hspace{1em}- \frac12 \int_{\eps y}^a \{q_+(u^\eps)-q_-(u^\eps)\}(\xi) v^\eps_\xi(\xi)\, d\xi - \int_{\eps y} ^a \xi (\eta(u^\eps))_\xi(\xi)\, d\xi\\
	&= -\eps\int_{\eps y}^a \eta''(u^\eps) (u^\eps_\xi)^2 d\xi \leq 0,
\end{align*}
since all the $a$ under consideration satisfy \eqref{hyp:ordering}. The proposed inequality immediately recasts in the terms of the rescaled profile $\calU^\eps$ and $\calV^\eps$ as follows:
\begin{align*}
&\eta(\calU^\eps)_y(y) - \eps \eta(u^\eps)_\xi(a)\\
	&\hspace{1em}- \tfrac12\{(1-\calV^\eps)q_-(\calU^\eps) + (1+\calV^\eps)q_+(\calU^\eps)\}(y)
	+ \tfrac12\{(1-v^\eps)q_-(u^\eps) + (1+v^\eps)q_+(u^\eps)\}(a)\\
	&\hspace{1em}- \frac12 \int_{y}^{a/\eps} \{q_+(\calU^\eps)-q_-(\calU^\eps)\}(s) \calV^\eps{}'(s)\, ds - \int_{\eps y} ^a \xi (\eta(u^\eps))_\xi(\xi)\, d\xi \leq 0.
\end{align*}
{Let  $\delta > \eps y$ be given} and average the above inequality for $a\in (\delta, 2\delta)$
\begin{align*}
& \eta(\calU^\eps)_y (y) - \tfrac12\{(1-\calV^\eps)q_-(\calU^\eps) + (1+\calV^\eps)q_+(\calU^\eps)\}(y)\\
	&\hspace{1em} - \frac 1\delta \int_{\delta}^{2\delta} \eps\eta(u^\eps)_\xi(a)\, da + \frac 1{2\delta}\int_{\delta}^{2\delta}\{(1-v^\eps)q_-(u^\eps) + (1+v^\eps)q_+(u^\eps)\}(a)\, da\\
	&\hspace{1em}- \frac 1{2\delta} \int_\delta^{2\delta}\int_{y}^{a/\eps} \{q_+(\calU^\eps)-q_-(\calU^\eps)\}(s) \calV^\eps{}'(s)\, ds\, da\\
	&\hspace{1em}- \frac 1{2\delta}\int_\delta^{2\delta}\int_{\eps y} ^a \xi (\eta(u^\eps))_\xi(\xi)\, d\xi\, da \leq 0,
\end{align*}
which we rewrite with clear notations as:
\begin{multline*}
 \eta(\calU^\eps)_y (y) - \tfrac12\{(1-\calV^\eps)q_-(\calU^\eps) + (1+\calV^\eps)q_+(\calU^\eps)\}(y)
- A_1^{\eps,\delta} + A_2^{\eps,\delta} -  A_3^{\eps,\delta} -  A_4^{\eps,\delta} \leq 0.
\end{multline*}
We show hereafter how to handle $A_i^{\eps,\delta}$ first in the limit $\eps\to 0$, $\delta$ being fixed and then in the limit $\delta\to 0$. We now propose to show that
\begin{align}
&\lim_{\eps\to 0} A_1^{\eps,\delta}=0,\label{limitA1}\\
&\lim_{\delta\to 0}\lim_{\eps\to 0} A_2^{\eps,\delta}=q_+(u_+),\label{limitA2}\\
&\lim_{\delta\to 0}\lim_{\eps\to 0} A_4^{\eps,\delta}=0,\label{limitA4}\\
&\lim_{\eps\to 0} A_3^{\eps,\delta}= \int_y^{+\infty}(q_+(\calU)-q_-(\calU))\calV_y\, ds.\label{limitA3}
\end{align}

\noindent$\bullet$
Let us first consider
\begin{equation*}
A_1^{\eps,\delta} = \frac 1\delta \int_\delta^{2\delta} \eps \eta(u^\eps)_\xi (a) \, da = \frac \eps\delta (\eta(u^\eps)(2\delta) - \eta(u^\eps)(\delta)).
\end{equation*}
Thanks to the uniform sup norm estimate for $u^\eps$, we propose the rough estimate:
\begin{equation*}
|A_1^{\eps,\delta}|\leq \frac {2\eps}\delta \max_{|u|\leq \|u_0\|_{L^\infty}}|\eta(u)|,
\end{equation*}
which yields \eqref{limitA1}, the small parameter $\delta$ being fixed.

\noindent$\bullet$ 
 The convergence property stated in \cite{Boutin:2011la} immediately gives:
\begin{align*}
\lim_{\eps\to 0^+} A_2^{\eps,\delta} = \frac 1{2\delta} \int_\delta^{2\delta} \{(1-v)q_-(u) + (1+v)q_+(u)\}(a)\, da= \frac 1\delta \int_\delta^{2\delta} q_+(u)(a)\, da,
\end{align*}
since for any given fixed $a>0$, $v(a)$ boils down to $+1$. The total variation of $u$ being bounded, $u$ admits left and right traces everywhere and we thus have \eqref{limitA2}.

\noindent$\bullet$ 
Let us now handle the limit in $A_4^{\eps,\delta}$ considering the following identity:
\begin{equation*}
\int_\delta^{2\delta} \xi (\eta(u^\eps))_\xi \,d\xi = a \eta(u^\eps) (a) - \eps y \eta(u^\eps) (\eps y) - \int_{\eps y}^a \eta(u^\eps) (\xi)\,d\xi,
\end{equation*}
so that
\begin{align*}
A_4^{\eps,\delta} &= \frac1\delta \int_\delta^{2\delta} \int_{\eps y}^a \xi (\eta(u^\eps))_\xi\, d\xi\, da\\
& = \frac 1\delta \int_\delta^{2\delta} a \eta(u^\eps(a))\, da - \eps y \eta(u^\eps(\eps y)) - \frac 1\delta \int_\delta^{2\delta} \int_{\eps y}^a \eta(u^\eps (\xi))\, d\xi\, da,
\end{align*}
so that direct calculations give the following crude upper bound:
\begin{equation*}
|A_4^{\eps, \delta}| \leq \max_{|u|\leq \|u_0\|_{L^\infty}} |\eta(u)| \Big\{\dfrac{3\delta}{2} + \eps y +  (\dfrac{3\delta}{2} -\eps y)\Big\} \leq 3\delta \max_{|u|\leq \|u_0\|_{L^\infty}} |\eta(u)|.
\end{equation*}
We thus have \eqref{limitA4}.

\noindent$\bullet$ 
Let us now evaluate $A_3^{\eps,\delta}$ in the limit $\eps\to 0$. {Let us consider} the following decomposition:
\begin{multline}
\label{decomp3}
\int_y^{a/\eps} (q_+(\calU^\eps) - q_-(\calU^\eps)) \calV^\eps_y\,ds = 
\int_y^{a/\eps} (q_+(\calU) - q_-(\calU)) \calV_y\,ds\\
+ \int_y^{a/\eps} (q_+(\calU) - q_-(\calU)) (\calV^\eps_y-\calV)\,ds\\ 
+ \int_y^{a/\eps} \Big\{(q_+(\calU^\eps) - q_-(\calU^\eps)) - (q_+(\calU) - q_-(\calU))\Big\} \calV^\eps_y\,ds.
\end{multline}
The limit profile $\calU$ being bounded, we observe that for some constant $C>0$ independent of $\eps$
\begin{equation*}
|(q_+(\calU)-q_-(\calU))(s)\calV_y(s)| \leq C e^{-s^2/2},
\end{equation*}
so that the first term in \eqref{decomp3} is finite and reads in the limit:
\begin{equation*}
\lim_{\eps\to 0} \int_y^{a/\eps} (q_+(\calU) - q_-(\calU)) \calV_y\,ds = \int_y^{+\infty} (q_+(\calU) - q_-(\calU)) \calV_y\,ds.
\end{equation*}
Then we have, concerning the second term in \eqref{decomp3}:
\begin{align*}
\left|\int_y^{a/\eps} (q_+(\calU) - q_-(\calU)) (\calV^\eps_y-\calV)\,ds\right|& \leq \max_{|u|\leq \|u_0\|_{L^\infty}}|(q_+-q_-)(u)|\ \|\calV^\eps_y-\calV_y\|_{L^\infty(y,+\infty)}\\& \leq O(\eps),
\end{align*}
since direct calculations ensure the pointwise estimate
\begin{equation*}
|\calV^\eps_y(s)-\calV_y(s)| \leq C e^{-L^2/2\eps}e^{-s^2/2},\quad s\in\RR.
\end{equation*}
At last and for any given $Y>y$, the last integral term in \eqref{decomp3} can be conveniently estimated according to:
\begin{align}
&\left|\int_y^{a/\eps} \Big\{(q_+-q_-)(\calU^\eps)-(q_+-q_-)(\calU)\Big\} \calV^\eps_y\,ds\right|\notag\\
&\leq \int_y^{\min(a/\eps,Y)}\!\! |(q_+-q_-)(\calU^\eps)-(q_+-q_-)(\calU)|\,ds + \max_{|u|\leq \|u_0\|_{L^\infty}}\!\! |q_+-q_-|(u)\ \int_{\min(a/\eps,Y)}^{a/\eps} |\calV^\eps_y|\, ds,\label{termlimit3}
\end{align}
since $|\calV^\eps_y(s)|\leq 1$. Observe then that the first integral concerns a {compact interval $[y,\min(a/\eps,Y)]$} for all $\eps>0$ and $a\geq \eps y$ so that the $L^1_{\rm loc}$ convergence of $\calU^\eps$ towards $\calU$ ensures
\begin{equation*}
\lim_{\eps\to 0} \int_y^{\min(a/\eps,Y)} |(q_+-q_-)(\calU^\eps)-(q_+-q_-)(\calU)|\,ds = 0.
\end{equation*}
Then observe that the last integral in the right hand side of \eqref{termlimit3} satisfies the following crude estimate (at least for $\eps$ sufficiently small)
\begin{equation*}
\int_{\min(a/\eps,Y)}^{a/\eps} |\calV^\eps_y(s)|\, ds \leq \int_{Y}^{+\infty} |\calV^\eps_y(s)|\, ds = 1 - \calV(Y) \leq C e^{-Y^2/2},
\end{equation*}
so that for all $Y\in\RR,$ $Y>y$,
\begin{equation*}
\lim_{\eps \to 0} \int_{\min(a/\eps,Y)}^{a/\eps} |\calV^\eps_y(s)|\, ds \leq C e^{-\tfrac12 Y^2}.
\end{equation*}
As  a consequence,
\begin{equation*}
\lim_{\eps \to 0} \left| \dfrac{1}{\delta}\int_\delta^{2\delta}\int_{\min(a/\eps,Y)}^{a/\eps} |\calV^\eps_y(s)|\, ds\, da \right| \leq C e^{-\tfrac12 Y^2},
\end{equation*}
for all $Y\in\RR$, the best estimate being obtained when sending $Y$ to $+\infty$. As a conclusion we get the limit \eqref{limitA3} and the main result follows.
\end{proof}

Equipped with the entropy inequalities \eqref{lm:entropR} and \eqref{lm:entropL} satisfied along the layer profile, let us infer Proposition~\ref{prop:convex}.

\begin{proof}[Proof of Proposition~\ref{prop:convex}]
Observe first that in the case of a trivial relaxation layer, i.e. verifying $\calU(y) = \calU_{-\infty} = \calU_{+\infty}$, for all $y$ in $\RR$, the inequality \eqref{lm:entropR} reads:

\begin{multline}
q_+(u_+) \leq \frac12\big [(1-\calV(y))q_-(\calU_{+\infty})+(1+\calV(y))q_+(\calU_{+\infty})\big]\\
+\frac12 \big(q_+(\calU_{+\infty})-q_-(\calU_{+\infty})\big)\int_y^{+\infty}\calV(s)\,ds,
\end{multline}
for all $y>0$. One can check that $\int_y^{+\infty} \calV(s)\, ds$ tends to zero as $y$ goes to infinity.
Sending thus $y$ to $+\infty$ yields the required result \eqref{interf-steadyCLR}. The companion inequality \eqref{interf-steadyCLL} follows from \eqref{lm:entropL} with straighforward modifications.\\
Let us now consider the case of a non trivial relaxation layer, i.e. with $\calU_y(y)\neq 0$ for all finite $y$ in $\RR$. Let us recall from Theorem~\ref{thm:RelaxLayer} that the asymptotic conditions \eqref{eq:condasymptR}--\eqref{eq:condasymptL} must be met expressing the property that $\calU_y(y)$ vanishes as $|y|$ goes to infinity. Passing to the limit $y\to +\infty$ in the inequality \eqref{lm:entropR} (respectively \eqref{lm:entropL}) then gives the expected result \eqref{interf-steadyCLR} (resp. \eqref{interf-steadyCLL}) in view of the properties of the profiles $\calU(y)$ and $\calV(y)$ stated in Theorem~\ref{thm:RelaxLayer}.
\end{proof}

Let us again stress that the asymptotic properties \eqref{eq:condasymptR} and \eqref{eq:condasymptL} are also responsible for the validity of the inequalities \eqref{interf-steadyCLR}--\eqref{interf-steadyCLL} in the case of a non trivial relaxation layer.

To conclude this section, it suffices to prove the matching properties \eqref{eq:monotL} and \eqref{eq:monotR}.

\begin{proof}[Proof of Theorem~\ref{thm:matchCond}]
The monotonicity property of the smooth solution $u^\eps$, $\eps>0$ being fixed, reads also $(u_R-u_L) u^\eps_\xi \geq 0$, which we integrate for $\xi \in (\eps y, a)$ with fixed $y\geq 0$ and $a>0$ with $\eps y <a$ to get
$(u_R-u_L)(u^\eps(a)-u^\eps(\eps y))\geq 0$, that is $(u_R-u_L)(u^\eps(a)-\calU^\eps(y))\geq 0$. 
{Let $\delta >\eps y$ be given}, the previous inequality once integrated for $a\in (\delta, 2\delta)$ yields
$(u_R-u_L)(\frac 1\delta \int_\delta^{2\delta} u^\eps(a) \, da -\calU^\eps (y)) \geq 0$, sending $\eps$ to zero gives
$(u_R-u_L)(u_+-\calU(y))\geq 0$, for all $y>0$. Passing to the limit $y\to+\infty$ gives the required result \eqref{eq:monotR}.\\
Similar steps apply to get the companion inequality \eqref{eq:monotL}.
\end{proof}


\section{The coupling of genuinely nonlinear fluxes}
\label{sec:GNL}


\subsection{Overview of the results}

As put forward in the previous section (see for example Fig.~\ref{fig:doublewell}), the existence of several sonic points is clearly responsible for non trivial and exotic self-similar solutions. They make tedious the characterization of all possible CRD solutions. 
From now on, we shall restrict our attention to flux functions having a single sonic point and we will prove that the ordering of these sonic points plays a central role in the structure and in the multiplicity of CRD solutions. For definiteness, we consider flux functions $f_-$ and $f_+$ that are strictly convex. The sonic point of $f_-$ (respectively $f_+$) will be denoted $u_-^\star$ (resp. $u_+^\star$) with $u_-^\star \geq -\infty$ and finite as soon as $\lim_{u\to-\infty} f_-(u) = +\infty$.

It is worth to briefly restate the result of Section~\ref{sec:outer} in the case of a pair of strictly convex fluxes. In the half lines $\{\xi<0\}$ and $\{\xi>0\}$, $u$ can be made of at most a single wave, namely a rarefaction or an entropy shock. Due to the monotonicity of the solution, if two single waves coexist, then they are necessarily of the same type, {i.e.} both simultaneously shock waves or both simultaneously rarefaction waves. In any situation, an extra standing discontinuity at $\{\xi=0\}$ may be involved. More precisely, we prove the following result.

\begin{corollary}[Convex setting]
\label{cor:convex}
Let $f_-$ and $ f_+$ be two convex flux functions and {$(u_L,u_R)$ be Riemann data.} 

\begin{itemize}

\item Assume $u_L<u_R$: No standing shock for either the left flux $f_-$ or the right flux $f_+$ can stick on the interface. If a non trivial inner profile exists, the following matching conditions are in order:
\begin{equation}
\calU_{-\infty} = u_- = u^\star_- < \calU_{+\infty} = u_+ = u^\star_+.
\end{equation}

\item Assume $u_L>u_R$: The outer solution may be discontinuous at the interface with $u_- > u_+$. Its interplay with the inner solution comes as follows.
\smallskip
\begin{itemize}
\item A standing entropy satisfying shock either for the left flux $f_-$ or for the right flux $f_+$ may exist but it cannot coexist with a non trivial inner profile: namely $\calU_{-\infty} = \calU_{+\infty} =: \calU_{\flat}$ with either $\calU_{\flat} > u_+$ and/or $u_- > \calU_{\flat}$. In the first case, necessarily $u^\star_+ > u_R$ while in the second case $u_L > u^\star_-$.
\smallskip
\item Right and left standing shocks may coexist but under the condition $u^\star_- > u^\star_+$ with the following data ordering $u_L > u^\star_- > u^\star_+ > u_R$. In this situation the outer solution coincides with the initial data.
\smallskip
\item The existence of a non trivial inner profile can only arise with $\calU_{-\infty} = u_L > \calU_{+\infty} = u_R$ under the condition $u_L \geq u_- \geq u^\star_-$ and $u^\star_+ > u_+ \geq u_R$. The outer solution coincides with the initial data.
\end{itemize}

\end{itemize}
\end{corollary}

\begin{proof}

The proof of this result is a consequence of Theorem~\ref{thm:matchCond}. Indeed and for instance, the case of a self-similar solution with $\calU_{+\infty}\neq u_+$ (corresponding to a standing shock for $f_+$ sticked at the interface) comes with the strict inequalities \eqref{eq:RLconsR}, namely:
\begin{equation}
\label{interfCLR}
f_+'(\calU_{+\infty}) >0, \quad f_+'(u_+)<0.
\end{equation}
In other words, $u_+<u^\star_+<\calU_{+\infty}$.

Correspondingly, the non matching property $\calU_{-\infty}\neq u_-$ implies that
\begin{equation}
\label{interfCLL}
f_-'(\calU_{-\infty}) <0, \quad f_-'(u_-)>0,
\end{equation}
that is $\calU_{-\infty}<u^\star_- < u_-$.

In both cases observe that necessarily $u_L>u_R$. The monotonicity properties \eqref{eq:monotL} and \eqref{eq:monotR} show that these two inequalities may coexist simultaneously. Since the solution is decreasing $\calU_{+\infty} \leq \calU_{-\infty}$, we necessarily have in this last case $u^\star_+ < u^\star_-$.

As a consequence of \eqref{eq:nontrivial}, the validity of only one of the two sets of inequalities \eqref{interfCLR} or \eqref{interfCLL} suffices to imply that the inner coupling profile $\calU(y)$ is constant, namely $\calU_{+\infty}=\calU_{-\infty}$.
Let us denote  $\calU = \calU_{+\infty}=\calU_{-\infty}$ this common value.

Figure~\ref{fig:22} displays the typical configuration in which both relations \eqref{interfCLR} and \eqref{interfCLL} coexist, and highlight the fact that infinitely many configurations satisfying \eqref{interfCLR} and \eqref{interfCLL} may be built.

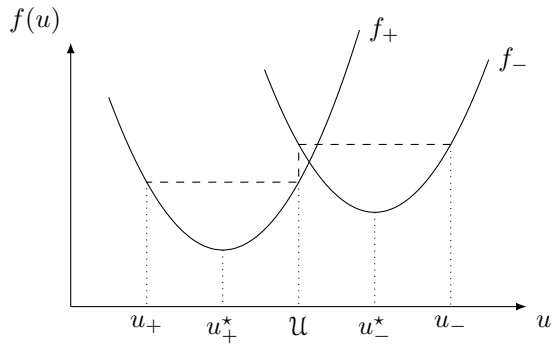
\begin{figure}[ht!]
\begin{center}
\begin{tikzpicture}[>=latex,scale=1]
 \draw[->] (-3,0) -- (3,0) node[below right] {$u$};
 \draw[->] (-3,0) -- (-3,3.5) node[above left] {$f(u)$};
 \begin{scope}
 \draw plot[domain=-0.45:2.5,smooth] (\x,{1.25+0.9*(\x-1)*(\x-1)}) node[right] {$f_-$} ;
 \draw plot[domain=-2.5:0.8,smooth] (\x,{0.75+0.9*(\x+1)*(\x+1)}) node[right] {$f_+$} ;
 \end{scope}
 \begin{scope}[dotted]
 \draw (1,1.25) -- (1,0) node[below] {$u^\star_-$};
 \draw (-1,0.75) -- (-1,0) node[below] {$u^\star_+$};
 \end{scope}
 \draw[dashed] (2,2.15) -- (0,2.15) -- (0,1.65) -- (-2,1.65);
 \begin{scope}[dotted]
 \draw (2,2.15) -- (2,0) node[below] {$u_-$};
 \draw (0,1.65) -- (0,0) node[below] {$\calU$};
 \draw (-2,1.65) -- (-2,0) node[below] {$u_+$};
 \end{scope}
\end{tikzpicture}
\end{center}
\caption{Non matching property}
\label{fig:22}
\end{figure}

Next, a more familiar situation can be obtained when assuming the set of inequalities \eqref{interfCLR} to be valid and the inequalities in \eqref{interfCLL} to be wrong (or conversely \eqref{interfCLL} to hold and \eqref{interfCLR} to fail) which corresponds to the matching property $u_-=\calU_{-\infty}$. In other words a pure stationary shock for the flux $f_+$ is sticked to the interface (resp. a pure stationary shock for the flux $f_-$ stands at the interface).

The Corollary~\ref{cor:MR3} gives in fact the following: assuming \eqref{interfCLR} to be satisfied necessarily implies that $f'(u(s))<0,\ \forall s>0,\ u(s)=u_R,\ \forall s>0$. Conversely assuming \eqref{interfCLL} to be true gives that  $u(s)=u_L$ for all $s<0$.

Let us investigate the existence of a relaxation profile in the case of a perfect matching on the left and on the right, namely $u_-=\calU_{-\infty}$ and $u_+=\calU_{+\infty}$. Let us investigate, under this assumption, the existence of a non trivial coupling profile. Under the matching property, the existence condition requires the inequalities $f_-'(u_-)\geq 0$ and $f_+'(u_+)\leq 0$ to be fullfilled. Assuming first a decreasing initial data, namely $u_L>u_R$, the solution decreases and the above inequalities imply $u(s)=u_L,\ s<0$ since $f_-'(u(s))\geq f'_-(u_-)\geq 0$ and $u(s)=u_R,\ s>0$.
The initial data satisfying $f_+'(u_R)\leq 0$ and $f_-'(u_L)\geq 0$, then necessarily one has $u_L\geq u_-^\star$ and $u_+^\star \geq u_R$.

Conversely assuming $u_L<u_R$, the self-similar solution increases and is smooth in the half lines $\{\xi<0\}$ and $\{\xi>0\}$, involving only rarefaction fans. Assuming $u_- > u_L$, that is $f'_-(u(s))<0$ for $s<0$ small enough, then necessarily $f'_-(u_-) = 0$ so that $u_- = u^\star_-$. Similarly, one has $u_+ = u^\star_+$ with $u_- < u_+$ that is: $u^\star_- < u^\star_+$ is a necessary condition.

\end{proof}


\subsection{A partial selection criterion}

In the last proof (see also Figure~\ref{fig:22}), it appears that in the case $u^\star_+<u^\star_-$, two distinct families of possible self-similar solutions depending on a real parameter may arise. Such continuum of solutions arise for pairs of states $(u_L,u_R)$ with either $u_L<u^\star_-$ and $u^\star_+<u_R$, or with $u_L>u^\star_+$ and $u^\star_->u_R$. Thanks to the results in the previous sections it is possible now to characterize entirely the possible value of $\bar u$ (among a continuum) that may appear in each of these CRD solutions:

\begin{proposition}[Uniqueness of the double-waved CRD solutions in the convex case]
Let $f_-$ and $f_+$ be two strictly convex flux functions with respective sonic points $u^\star_-$ and $u^\star_+$ (finite or not).
Let $u$ be a CRD solution to~\eqref{eq:left}--\eqref{eq:right} with Riemann data $(u_L,u_R)$, in the class of solutions consisting of a left-wave in the half-space $\{\xi<0\}$ followed by a constant state $\bar u$ across the interface, and a right-wave in the half-space $\{\xi>0\}$, then the intermediate state $\bar u\in(u^\star_+,u^\star_-)$ comes uniquely defined by the selection criterion \eqref{eq:selection}.
\end{proposition}

\begin{proof}
Due to the monotonicity property of the solution and to the strict convexity of the flux functions, the solution $u$ consists in two waves of the same kind: either a rarefaction wave for the flux function $f_-$ connecting $u_L$ to $\bar u$ followed by another rarefaction wave for the flux function $f_+$ connecting $\bar u$ to $u_R$, or a shock wave for the flux function $f_-$ connecting $u_L$ to $\bar u$ followed by another shock wave for the flux function $f_+$ connecting $\bar u$ to $u_R$. 
Figures~\ref{fig-dblraref} and~\ref{fig-dblshock} represent these two situations, with the solution $u$ in the $x-t$ plane on the left picture, and the subordinate function $h(\cdot;u)$ on the right picture. The characterization underlined in Corollary~\ref{cor:MR3} becomes somehow explicit and we detail hereafter the situation in each case.
\begin{itemize}
\item Suppose first $u_L<\bar u<u_R$. The CRD solution is a double rarefaction wave, as depicted in Figure~\ref{fig-dblraref}. {Let us denote} by $\lambda_-(\bar u):=f_-'(\bar u)<0$ and $\lambda_+(\bar u):=f_+'(\bar u)>0$ the characteristic velocities in between the two rarefaction waves.
Obviously, such a situation requires the following ordering to arise: $u^\star_+<\bar u<u^\star_-$. Using then the explicit expression \eqref{eq:hprime}, the condition \eqref{eq:selection} reads:
\begin{equation}
\label{eq:ubarraref}
\lambda_-(\bar u)^2 = \lambda_+(\bar u)^2,
\end{equation}
or equivalently  (remember $\lambda_-(\bar u ) <0 < \lambda_+(\bar u)$),
\begin{equation*}
\lambda_-(\bar u) + \lambda_+(\bar u) = 0.
\end{equation*}
It suffices now to observe that $\lambda_-(u^\star_-) + \lambda_+(u^\star_-)=f_+'(u^\star_-)>0$ and $\lambda_-(u^\star_+) + \lambda_+(u^\star_+)=f_-'(u^\star_+)<0$, so that the strict convexity of the sum $f_-+f_+$ suffices to conclude to the existence and uniquenesse of the solution $\bar u$ to \eqref{eq:ubarraref}.

\item Suppose now $u_R<\bar u<u_L$. The CRD solution is a double shock wave, as depicted in Figure~\ref{fig-dblshock}. {Let us denote}
\begin{equation}
\label{eq:shockvel}
\lambda_-(\bar u):=\frac{f_-(\bar u)-f_-(u_L)}{\bar u-u_L}<0,\textrm{ and }\lambda_+(\bar u):=\frac{f_+(\bar u)-f_+(u_R)}{\bar u-u_R}>0
\end{equation}
the velocities of both shock waves satisfying the Lax inequalities:
\begin{equation}
\label{eq:laxconv}
f_-'(u_L)>\lambda_-(\bar u)>f_-'(\bar u),\textrm{ and } f_+'(\bar u)>\lambda_+(\bar u)>f_+'(u_R).
\end{equation}
Obviously, such a situation requires again the following ordering to arise: $u^\star_+<\bar u<u^\star_-$. Using the explicit expression \eqref{eq:hprime}, the condition \eqref{eq:selection} reads:
\begin{equation}
\label{eq:ubarshock}
\lambda_-(\bar u)(\tfrac12\lambda_-(\bar u)-f_-'(\bar u)) = \lambda_+(\bar u)(\tfrac12\lambda_+(\bar u)-f_+'(\bar u)).
\end{equation}
{Let us define} for $\bar u\in(u^\star_+,u^\star_-)$ the quantity
\[\kappa(\bar u):=\lambda_-(\bar u)\big(\tfrac12\lambda_-(\bar u)-f_-'(\bar u)\big) - \lambda_+(\bar u)\big(\tfrac12\lambda_+(\bar u)-f_+'(\bar u)\big).\]
For the bound values, we get
\begin{equation}
\kappa(u^\star_+) = \lambda_-(u^\star_+)(\tfrac12\lambda_-(u^\star_+)-f_-'(u^\star_+)) - \tfrac12\lambda_+(u^\star_+)^2,
\end{equation}
\begin{equation}
\kappa(u^\star_-) = \tfrac12\lambda_-(u^\star_+)^2 - \lambda_+(u^\star_-)(\tfrac12\lambda_+(u^\star_-)-f_+'(u^\star_-)),
\end{equation}
and due to \eqref{eq:laxconv} and \eqref{eq:shockvel}, $\tfrac12\lambda_-(u^\star_+)-f_-'(u^\star_+)>-\tfrac12\lambda_-(u^\star_+)>0$ on the one hand, and $\tfrac12\lambda_+(u^\star_-)-f_+'(u^\star_-)<-\tfrac12\lambda_+(u^\star_-)<0$ on the other hand. Finally we get
\begin{equation}
\kappa(u^\star_+) <0 < \kappa(u^\star_-). 
\end{equation}
Suppose now for technical convenience that $f_-$ and $f_+$ are twice differentiable.
The function $\kappa$ is then differentiable on $(u^\star_+,u^\star_-)$, and for all $\check{u}\in(u^\star_+,u^\star_-)$ such that $u_R<\check{u}<u_L$ and \eqref{eq:shockvel}--\eqref{eq:laxconv} are both satisfied, we obtain
\begin{multline}
\kappa'(\check{u}) = \left(\lambda_-'(\check{u})(\lambda_-(\check{u})-f_-'(\check{u}))-\lambda_-(\check{u})f_-''(\check{u})\right) \\- \left(\lambda_+'(\check{u})(\lambda_+(\check{u})-f_+'(\check{u}))-\lambda_+(\check{u})f_+''(\check{u})\right).
\end{multline}
All occurring quantities are signed so that $\kappa'(\check{u})>0$ and we obtain existence and uniqueness of the solution $\bar u$ to \eqref{eq:ubarshock}.
\end{itemize}
\end{proof}

\begin{figure}[!ht]
\centering
\vspace{-3.3em}
\hfill
\subfloat[\label{fig-dblraref-u}]{%
\begin{tikzpicture}[>=latex,scale=1]
 \draw[->] (-2,0) -- (2,0) node[right] {$x$};
 \draw[->] (-2,0) -- (-2,2) node[right] {$t$};
 \foreach \angle in {50,54,...,70} {\draw (0,0) -- +(\angle:2);}
 \foreach \angle in {130,126,...,110} {\draw (0,0) -- +(\angle:2);}
 \node at (-1.3,0.8) {$u_L$};
 \node at (1.3,0.8) {$u_R$};
 \node at (0,1.5) {${\bar u}$};
 \node[below] at (0,0) {$0$};
\end{tikzpicture}
}
\hfill
\subfloat[\label{fig-dblraref-h}]{%
\begin{tikzpicture}[>=latex,scale=1.8]
 \draw[->] (-2,0) -- (2,0) node[right] {$\xi$};
 \draw[->] (0,-0.7) -- (0,1) node[right] {$h(\xi;u)$};
 \draw[domain=-2:-1,smooth] plot (\x,{\x*\x/2+\x+1/2-1/8});
 \draw[domain=-1:-1/2,smooth] plot (\x,-1/8);
 \draw[domain=-1/2:0,smooth] plot (\x,{\x*\x/2+\x/2});
 \draw[domain=0:1,smooth] plot (\x,{\x*\x/2-1*\x});
 \draw[domain=1:1.5,smooth] plot (\x,-1/2);
 \draw[domain=1.5:2,smooth] plot (\x,{\x*\x/2-1.5*\x-1/2-1.5*1.5/2+1.5*1.5});
 \draw[dashed] (-1,-1/8) -- (-1,0) node[right,rotate=60] {\small$f_-'(u_L)$};
 \draw[dashed] (-1/2,-1/8) -- (-1/2,0) node[right,rotate=60] {\small$f_-'(\bar u)$};
 \draw[dashed] (1,-1/2) -- (1,0) node[right,rotate=60] {\small$f_+'(\bar u)$};
 \draw[dashed] (3/2,-1/2) -- (3/2,0) node[right,rotate=60] {\small$f_+'(u_R)$};
 \node[above,fill=white,outer sep=0.4pt] at (0,0) {$0$};
\end{tikzpicture}
}
\hfill
\caption{Structure of double-rarefaction solutions}
\label{fig-dblraref}
\end{figure}
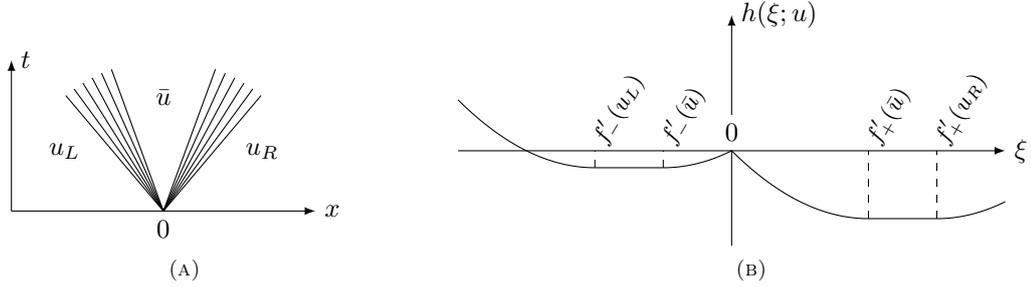
%
\begin{figure}[!ht]
\centering
\hfill
\subfloat[\label{fig-dblshock-u}]{%
\begin{tikzpicture}[>=latex,scale=1]
 \draw[->] (-2,0) -- (2,0) node[right] {$x$};
 \draw[->] (-2,0) -- (-2,2) node[right] {$t$};
 \foreach \angle in {60} {\draw (0,0) -- +(\angle:2);}
 \foreach \angle in {120} {\draw (0,0) -- +(\angle:2);}
 \node at (-1.3,0.8) {$u_L$};
 \node at (1.3,0.8) {$u_R$};
 \node at (0,1.5) {${\bar u}$};
 \node[below] at (0,0) {$0$};
\end{tikzpicture}
}
\hfill
\subfloat[\label{fig-dblshock-h}]{%
\begin{tikzpicture}[>=latex,scale=1.8]
 \draw[->] (-1.8,0) -- (1.8,0) node[right] {$\xi$};
 \draw[->] (0,-0.7) -- (0,0.8) node[right] {$h(\xi;u)$};
 \draw[domain=-1.5:-3/4,smooth] plot (\x,{\x*\x/2-3/4-1/4*(\x+3/4)});
 \draw[domain=-3/4:0,smooth] plot (\x,{\x*\x/2+\x});
 \draw[domain=0:1,smooth] plot (\x,{\x*\x/2-1.1*\x});
 \draw[domain=1:1.8,smooth] plot (\x,{\x*\x/2-1.1-1/4*(\x-1)});
 \draw[dashed] (1,-0.6) -- (1,0) node[right,rotate=60] {$\lambda_+(\bar u)$};
 \draw[dashed] (-3/4,-15/32) -- (-3/4,0) node[right,rotate=60] {$\lambda_-(\bar u)$};
 \node[above,fill=white,outer sep=0.4pt] at (0,0) {$0$};
\end{tikzpicture}
}
\hfill
\caption{Structure of double-shock solutions}
\label{fig-dblshock}
\end{figure}
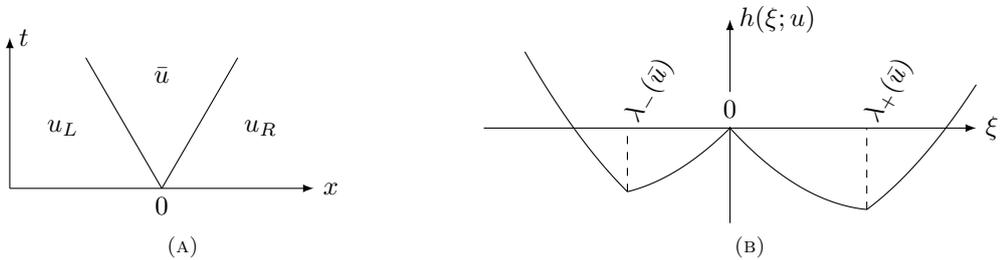

\begin{remark}
Whithout the convexity assumptions on the fluxes $f_-$ and $f_+$, the uniqueness result above is lost in general. Some numerical counterexamples are given in the last section of this paper.
\end{remark}


\subsection{Explicit analysis of the criterion for quadratic convex flux functions}

For the special case of very simple convex flux functions $f_-$ and $f_+$, {it is possible} to make explicit the intermediate value $\bar u$ located between two rarefaction waves or two shock waves (solution to \eqref{eq:ubarraref} or to \eqref{eq:ubarshock}). This is the aim of the following Corollary.
\begin{corollary}[Double-waved CRD solutions for convex quadratic fluxes]
\label{cor:quadratic}
{Let us fix} $c<0$ and consider the flux functions $f_-(u)=\tfrac12 u^2$ and $f_+(u)=\tfrac12 (u-c)^2$. We are concerned with Riemann data $(u_L,u_R)$.
\begin{itemize}
\item There is at most one double rarefaction CRD solution to \eqref{eq:left}--\eqref{eq:right}. The intermediate state in between the two waves is 
\begin{equation}
  \label{eq:dblraref}
\bar u =\frac c2.
\end{equation}
This solution may exist only under the constraint
\begin{equation}
\label{eq:ubarbound1}
u_L < \frac c2 < u_R.
\end{equation}
\item There is at most one double shock CRD solution to \eqref{eq:left}--\eqref{eq:right}. The intermediate state in between the two waves is 
\begin{equation}
  \label{eq:dblshock}
\bar u=\dfrac{u_R^2-u_L^2-4c^2}{2(u_R-u_L-4c)}.
\end{equation}
This solution may exist only under the set of constraints
\begin{equation}
\left.
\begin{aligned}
  2c<u_R<c/2, &\qquad u_R-u_L>2c, \label{eq:ubarbound}\\
  u_R-4c-2\sqrt{c(5c-2u_R)} < &u_L < u_R+2\sqrt{c(2u_R-c)}.
  \end{aligned}
\right.
\end{equation}
\end{itemize}
\end{corollary}

\begin{proof}
The sonic points are here $u^\star_+=c$ and $u^\star_-=0$ respectively. Starting from the equations~\eqref{eq:ubarraref} and \eqref{eq:ubarshock}, the results follow directly from simple algebraic calculations. The required inequalities \eqref{eq:ubarbound} (and similarly for \eqref{eq:ubarbound1} in the case of two rarefactions) are equivalent to the monotonicity property for $u$ (see~\eqref{eq:monotonicity}) and the natural ordering of velocities (see~\eqref{eq:shockvel} for example). The intermediate state $\bar u$ has to belong to the interval $[u_R,u_L]$, what restricts the domain of existence of such a double-waved solution. We leave these details to the reader.
\end{proof}


\subsection{Graphical overview for the quadratic convex case}

We present now two Figures~\ref{fig-pos} and~\ref{fig-neg} corresponding respectively to the cases $c>0$ and $c<0$. The main idea is to represent, in the plane of Riemann data $(u_L,u_R)$, the map of all possible CRD solutions $u$ to \eqref{eq:left}--\eqref{eq:right} satisfying the constraints we underlined previously. For the frontier curves to be explicit, we again restrict the analysis to the case of quadratic convex fluxes $f_-$ and $f_+$ used in the Corollary~\ref{cor:quadratic} above. Let us mention importantly that we did not prove that any of these solutions is effectively the limit of a subsequence of a viscous self-similar solution to~\eqref{eq:Dafermos}--\eqref{eq:Boundcond}. We only proceed by using necessary conditions. Of course, when the uniqueness occurs, then the exhibited solution has to be the unique CRD solution and to be indeed the limit of the considered vanishing process.
Let us now describe more into the details these results.

In Figure~\ref{fig-pos}, many (straight) curves represent either vanishing characteristic velocities or standing shock waves for either the left of the right problem. These curves, together with the usual line $u_R=u_L$ (monotonicity transition), divide the plane into 10 areas, named from letter A to letter J. In any of these area, the associated table gives the description of possible CRD solutions. In that description, $R_-$, $R_+$, $S_-$ and $S_+$ correspond respectively to a left-rarefaction wave, a right-rarefaction wave, a left-shock wave, a right-shock wave. Compound solutions may also involve internal transition wave, which named as $T$, for example $R_-TR_+$ is a solution consisting of a left-rarefaction wave, sticked to an interfacial transition wave, followed by a right-rarefaction wave.

In Figure~\ref{fig-neg}, due to the inversion {of the characteristic velocities $c<0$}, the characteristic boundaries of the areas are now more numerous, dividing the plan into 17 areas, named from letter A to letter Q. The first associated table gives again the description of possible CRD solutions in each area. For example, in the squared region $\{(u_L,u_R),\ c<u_L<c/2,\ c/2<u_L<0\}$, named G, three solutions satisfy all the considered contraints: a single left rarefaction wave $R_-$, a single right rarefaction wave $R_+$, and a double-rarefaction with intermediate state $\bar u=c/2$ ($R_-R_+$). 
The second table gives the reversed description of possible solutions, and the corresponding area of presence for each of them. As proved from Corollary~\ref{cor:quadratic} the double-waved solutions with some intermediate state, i.e. in families $S_-S_+$ and $R_-R_+$, are uniquely defined thanks to characterizations~\eqref{eq:dblraref} and~\eqref{eq:dblshock} respectively, therefore the corresponding frontier curves manifest the constraints~\eqref{eq:ubarbound1} and~\eqref{eq:ubarbound} respectively.


\clearpage

\begin{figure}[!ht]
   \centering
    \includegraphics[width=0.75\linewidth]{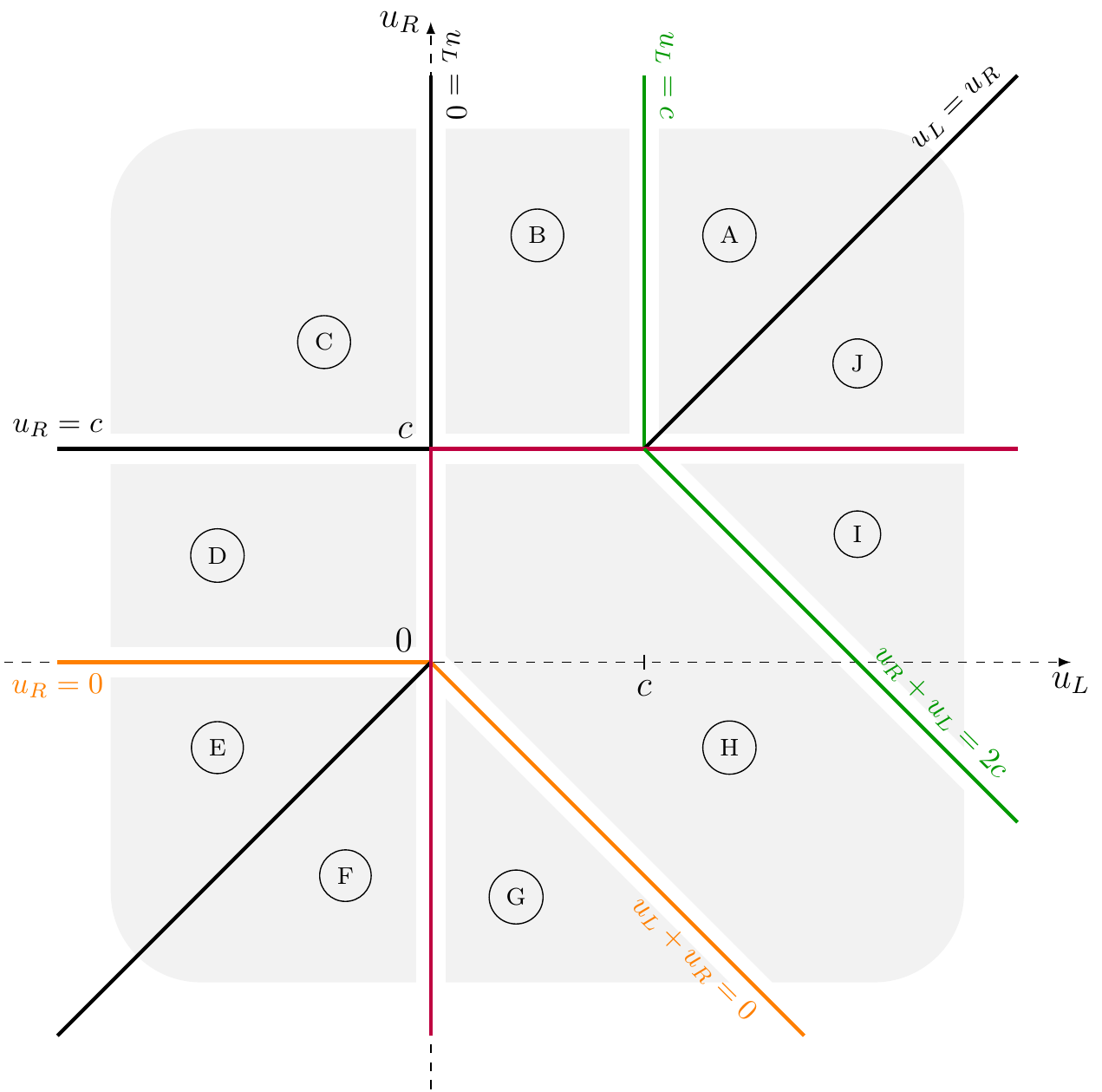}

    \vspace{1em}
    \begin{tabular}{|c|c|}
      \hline
      Area & Solutions\\
      \hline
      A & $R_+$\\
      B & $TR_+$\\
      C & $R_-TR_+$ \\
      D & $R_-T$\\
      E & $R_-$\\
      F & $S_-$\\
      G & $S_-$ or $T$\\
      H & $T$\\
      I & $S_+$ or $T$\\
      J & $S_+$\\\hline
    \end{tabular}
   \caption{Candidate CRD solutions for two convex quadratic fluxes (case $c>0$).}
    \label{fig-pos}
\end{figure}

\clearpage

\begin{figure}[!ht]
   \centering
      \includegraphics[width=0.75\linewidth]{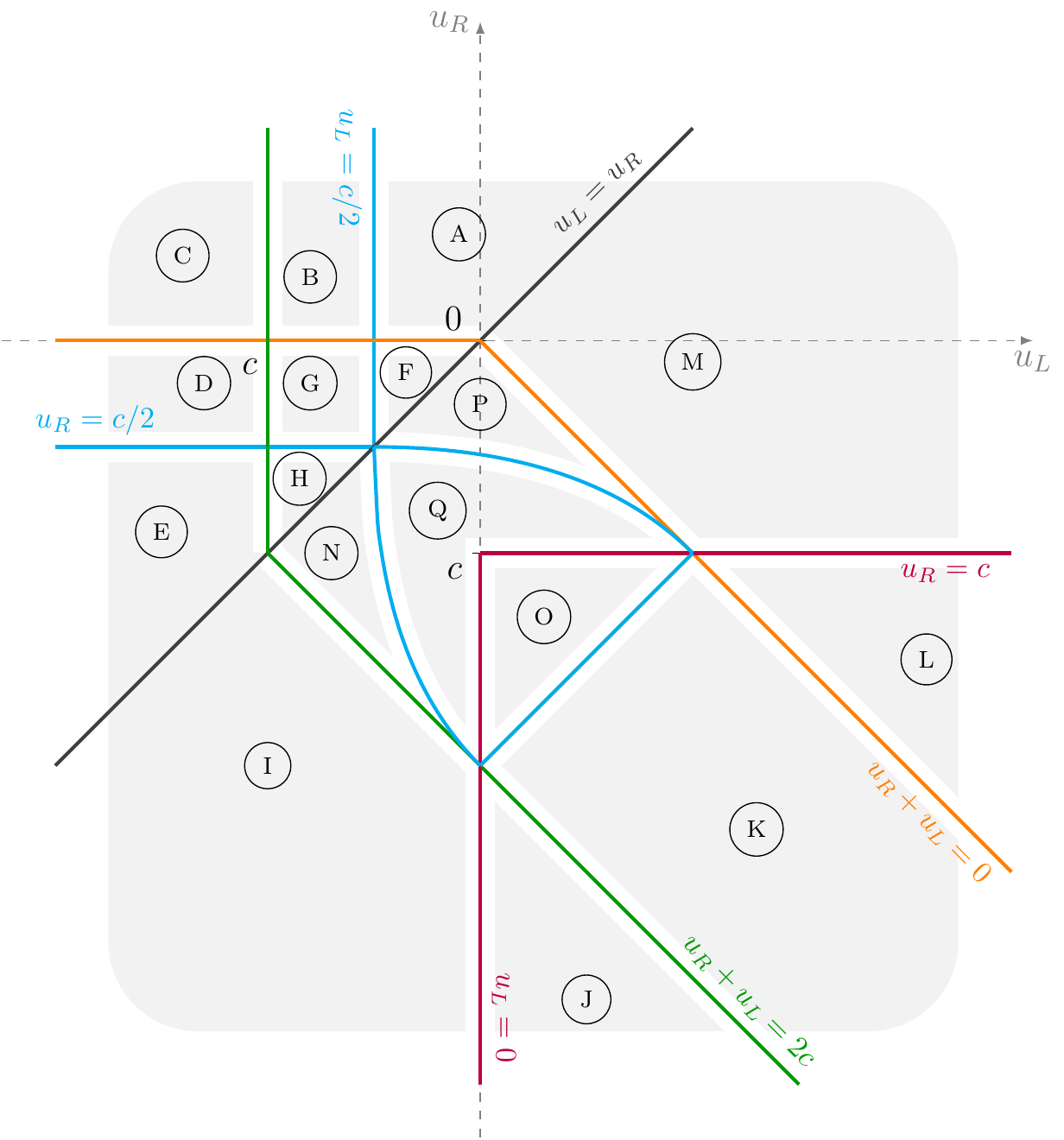}

      \vspace{1em}
      \begin{tabular}{|c|c|}
        \hline
        Area & Solutions\\
        \hline
        A & $R_+$\\
        B & $R_-R_+$ or $R_+$\\
        C & $R_-R_+$\\
        D & $R_-R_+$ or $R_-$\\
        E & $R_-$\\
        F & $R_-$ or $R_+$\\
        G & $R_-$ or $R_+$ or $R_-R_+$\\
        H & $R_-$ or $R_+$\\
      
        I & $S_-$\\
        J & $S_-$ or $T$\\
        K & $S_-$ or $S_+$ or $T$\\ 
        L & $S_+$ or $T$\\  
        M & $S_+$\\  
        N & $S_-$ or $S_+$\\
        O & $S_-$ or $S_+$ or $S_-S_+$ or $T$\\
        P & $S_-$ or $S_+$\\
        Q & $S_-$ or $S_+$ or $S_-S_+$\\\hline
      \end{tabular}
      \hspace{1cm}
      \begin{tabular}{|c|c|}
        \hline
        Solution & Areas\\
        \hline
        $R_+$ & A, B, F, G\\            
        $R_-$ & E, D, G, H\\            
        $R_-R_+$ & B, C, D, H\\         
        $S_-$ & I, J, K, N, O, P, Q\\   
        $S_+$ & K, L, M, N, O, P, Q\\   
        $S_-S_+$ & O, Q\\               
        $T$ & K, O\\                    
        \hline
      \end{tabular}
   \caption{Candidate CRD solutions for two convex quadratic fluxes (case $c<0$).}
   \label{fig-neg}
\end{figure}

\clearpage
\section{Numerical experiments}
\label{sec:numerics}

Through our previous analysis, the uniqueness results concerning the selection condition~\eqref{eq:MR3} are restricted to the class of double-waved solutions and to the coupling of strictly convex flux functions. To handle more general flux functions, we propose hereafter a brief numerical study, focusing on such class of solutions with no internal coupling layer. Our approach is based on the (discrete) Legendre-Fenchel transform that allows to compute an approximation of the self-similar entropy solution for both the left conservation law \eqref{eq:ESleft} and for the right conservation law \eqref{eq:ESright}, the intermediate state $\bar u$ being first arbitrarily given at the interface $\xi=0$ and then selected among possible values through the selection criterion~\eqref{eq:selection}. In the following lines, we first explain with more details our numerical strategy and then validate it in front of the quadratic case, with explicitly known solutions (from Corollary~\ref{cor:quadratic}). Finally we use this strategy to illustrate the existence of multiple double-waved CRD solutions for some non-convex fluxes. 


\subsection{Numerical strategy}

Let us first describe our numerical strategy. Being given three states $u_L,\ \bar u,\ u_R$, we are interested in computing, the entropy solution to~\eqref{eq:ESleft} and~\eqref{eq:ESright} satisfying the boundary conditions~\eqref{eq:BC}, such that the interface value equals $\bar u$ in a strong sense, i.e. for both left and right traces at $\xi=0$. Not any values of the parameter $\bar u$ are suitable, due to the coupling constraints, and characteristic velocities at the interface as well. We proceed using the brute-force method with many $\bar u$ values within the interval $[u_L,u_R]$. Let us consider only one half-problem, say with the flux $f_-$ and data $(u_L,\bar u)$. The corresponding Riemann entropy weak solutions may be obtained by using the classical convex or concave hulls for the corresponding flux function. The approximation of that solution is obtained by using the biconjugate $f_-^{\star\star}$ within the interval $[u_L,\bar u]$. This is done after introducing some small discretization parameters in the physical and dual spaces and assuming somehow the flux to be approximated thanks to piecewise affine flux function, as done by Dafermos in~\cite{Dafermos:72}. 
According to the relative position of $u_L$ and $u_R$ it is then possible to solve approximately either the equation~\eqref{eq:ubarraref} for double-rarefaction fans, or the equation~\eqref{eq:ubarshock} for double-shock solutions. After this, only some very few values of the intermediate parameter $\bar u$ do effectively solve the considered problem.


\subsection{Validation on the quadratic case}

The above strategy is first put to the test with the quadratic fluxes considered in Corollary~\ref{cor:quadratic} for the case $c=-1<0$.

\begin{itemize}
  \item Let us consider the double-shock CRD solution corresponding to the data $(u_L,u_R)=(0,-0.75)$. The expected value of $\bar u$, solution to~\eqref{eq:ubarshock}, is given from~\eqref{eq:dblshock} and approximately equals~$\bar u\simeq -0.528846...$. Thanks to our numerical procedure, we obtain the approximation $\bar u_{\rm approx}\simeq -0.528877...$.
   Figure~\ref{fig:DblShock} represents both the CRD solution (left) and the selection criterion~\eqref{eq:selection} (right), see also Figure~\ref{fig:hfunction}.
  
  \begin{figure}[!ht]
    \centering
    \includegraphics[trim=30 30 30 29,clip,width=0.48\linewidth]{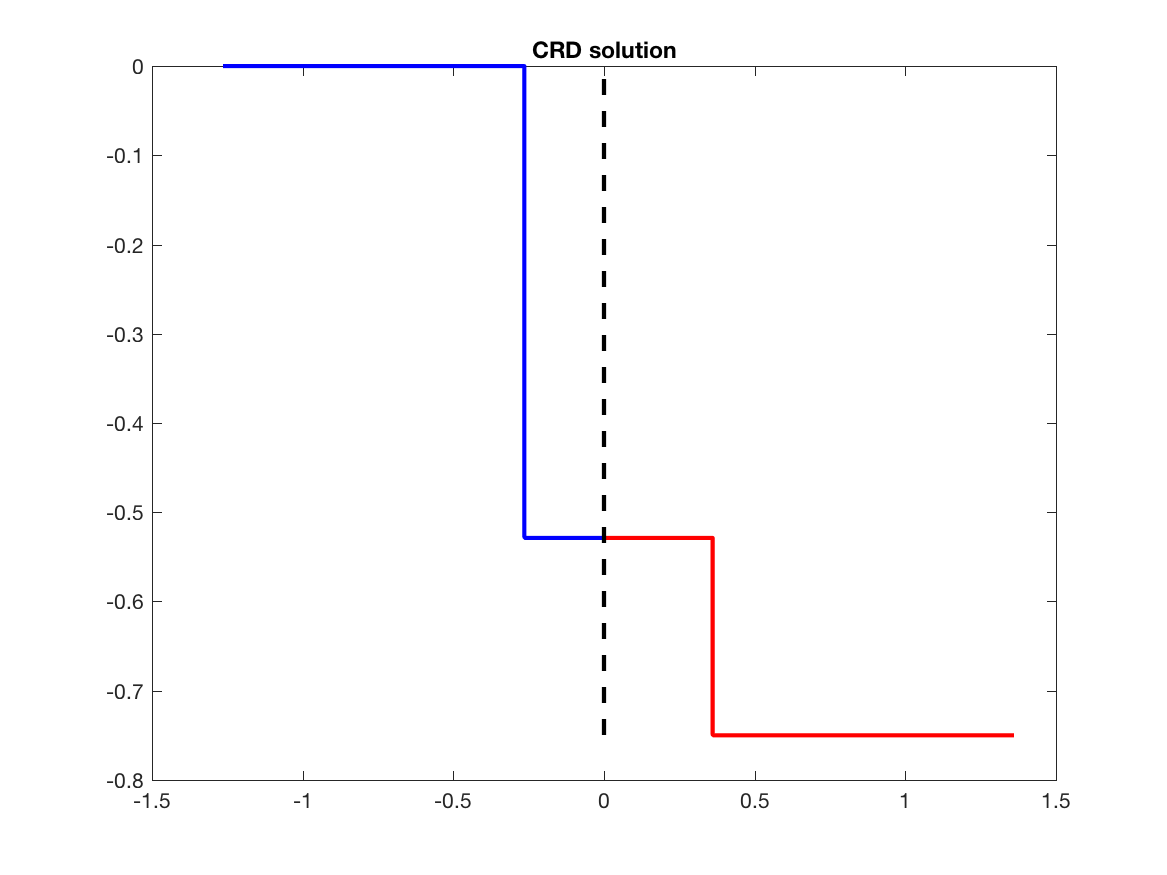}
    \includegraphics[trim=30 30 30 29,clip,width=0.48\linewidth]{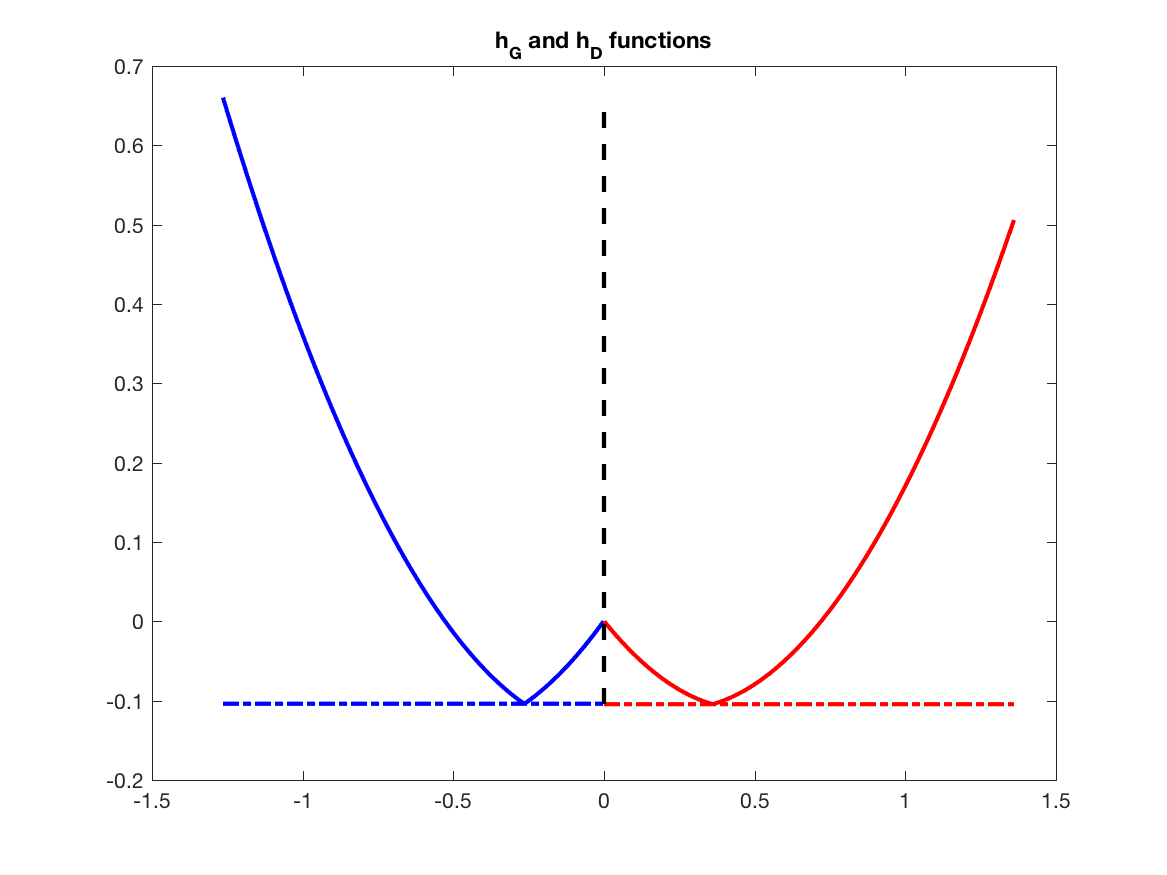}
    \caption{Numerical approximation of the double-shock solution. CRD solution $u$ (left) and selection criterion $h(\cdot;u)$ (right).}
     \label{fig:DblShock}
 \end{figure}
  
  \item Let us now consider double-rarefaction CRD solutions, with $(u_L,u_R)=(-0.7,0.1)$. The expected value of $\bar u$, solution to~\eqref{eq:ubarshock}, is given from~\eqref{eq:dblshock} and exactly equals~$\bar u = -0.5$. Thanks to our numerical procedure, we obtain the approximation $\bar u_{\rm approx}\simeq -0.49998...$. 
  Figure~\ref{fig:DblRaref} represents both the CRD solution (left) and the selection criterion~\eqref{eq:selection} (right), see also Figure~\ref{fig:hfunction}.
  
  \begin{figure}[!ht]
    \centering
    \includegraphics[trim=30 30 30 29,clip,width=0.48\linewidth]{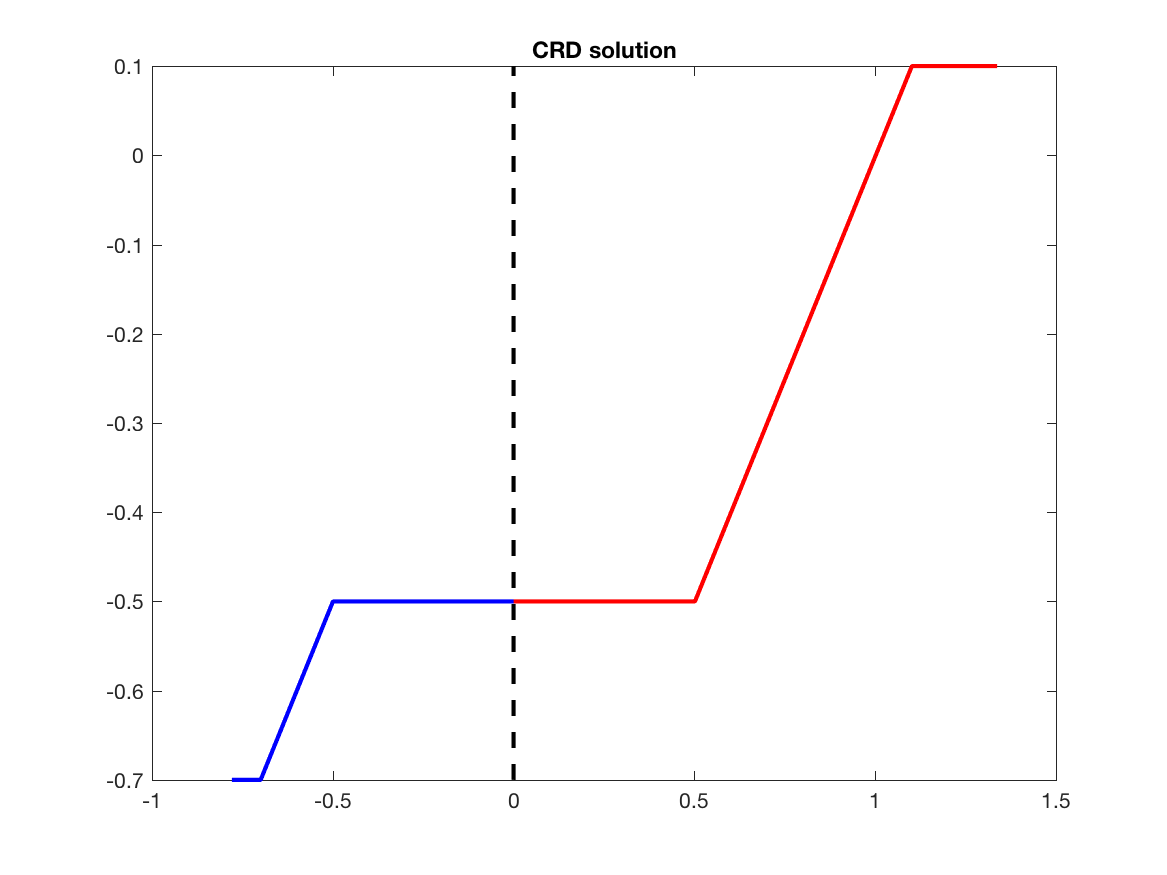}
    \includegraphics[trim=30 30 30 29,clip,width=0.48\linewidth]{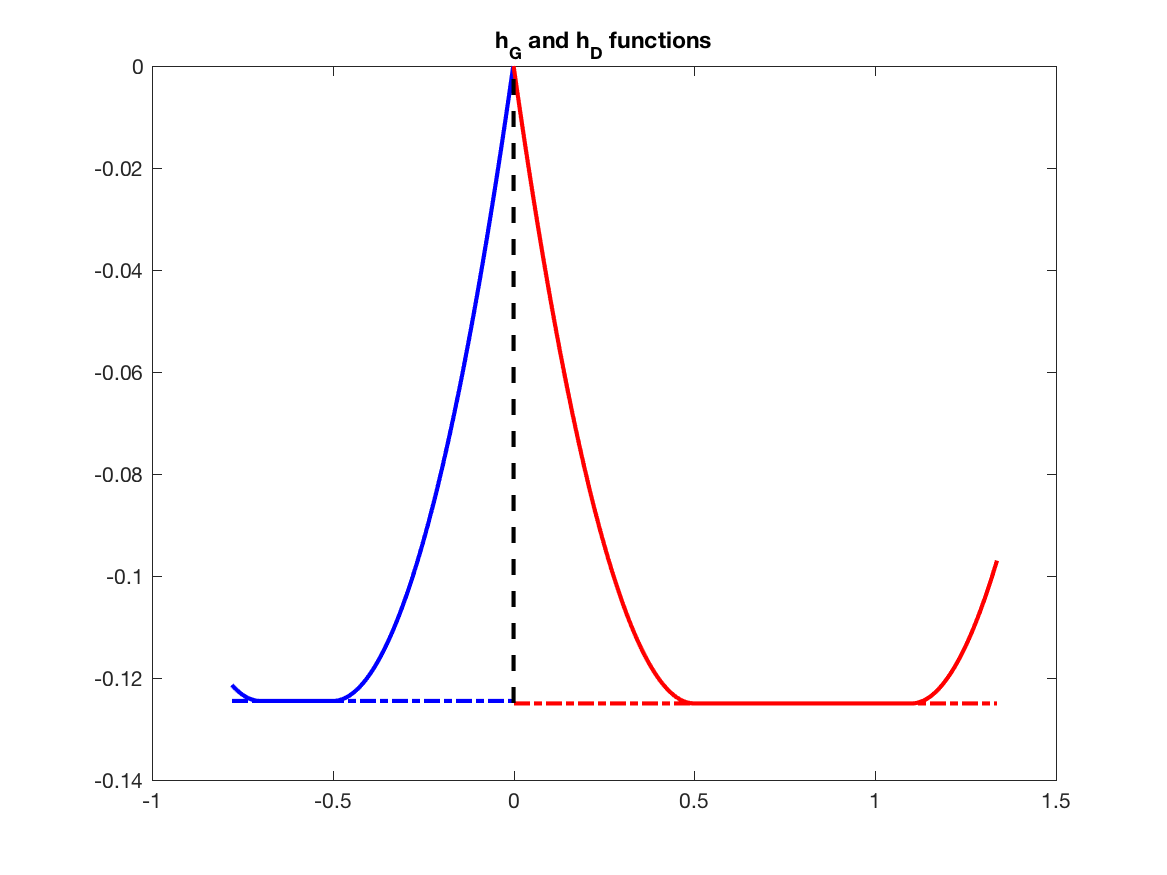}
    \caption{Numerical approximation of the double-rarefaction solution. CRD solution $u$ (left) and selection criterion $h(\cdot;u)$ (right).}
     \label{fig:DblRaref}
 \end{figure}

\end{itemize}


\subsection{Existence of multiple double waved CRD solution}

{Let us now consider} the following non-convex fluxes:
\begin{equation}
f_-(u) = \dfrac{u^4}{16} - \dfrac{u^2}{2} - u,\qquad f_+(u) = \dfrac{(u+1)^4}{16} - \dfrac{(u+1)^2}{2} + (u+1)
\end{equation}
Using the above numerical strategy, we focus now on the CRD solutions with Riemann data $(u_L,u_R) = (-2.5,1.5)$.
More precisely, we are interested in CRD solutions that consist of two separated wave fans for both the left- and the right-model and solve numerically the equation \eqref{eq:selection}. From the numerical solving of~\eqref{eq:selection}, three solutions appear. The corresponding intermediate values are approximately ${\bar u}\in\{-2.2231 , -0.5 , 1.2231\}$ respectively. 
Figure~\ref{fig:Multiple} represents any of these three solutions.

\begin{figure}[!ht]
  \centering
  \includegraphics[trim=30 30 30 29,clip,width=0.48\linewidth]{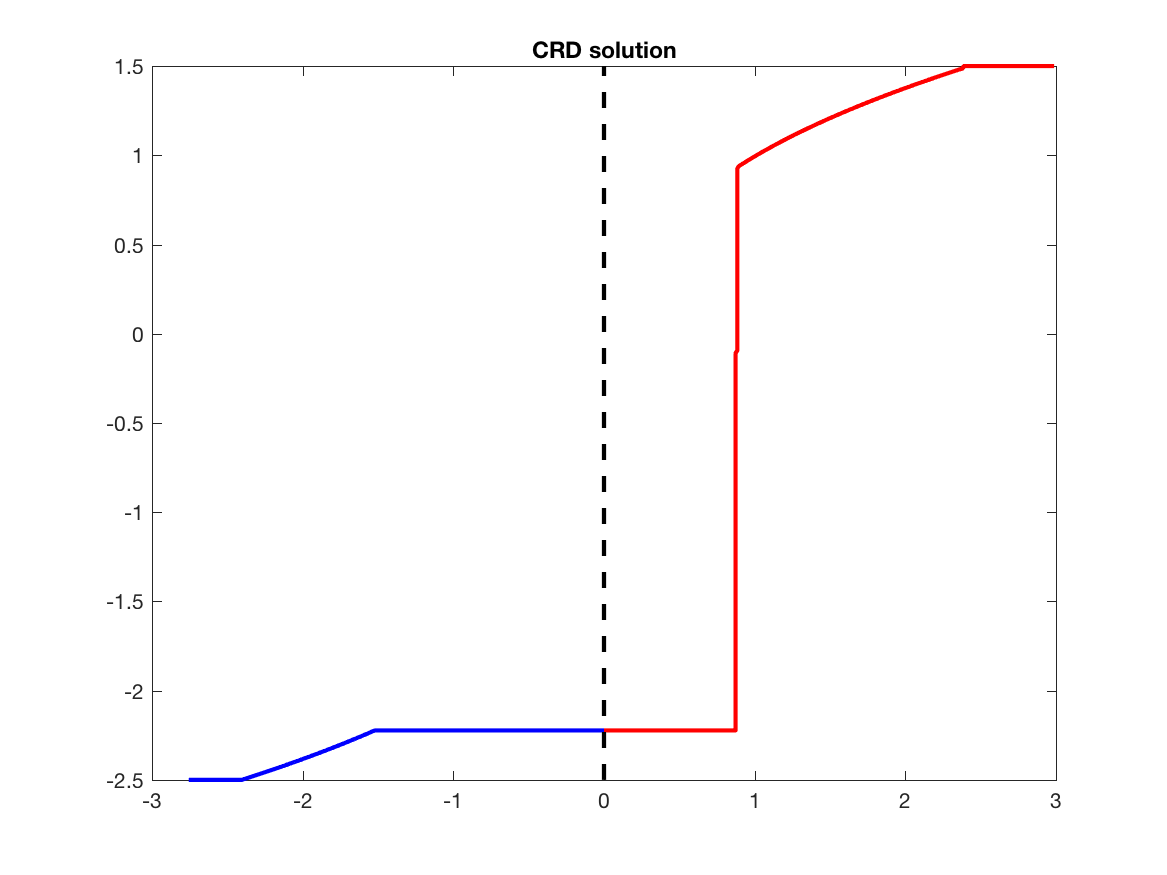}
  \includegraphics[trim=30 30 30 29,clip,width=0.48\linewidth]{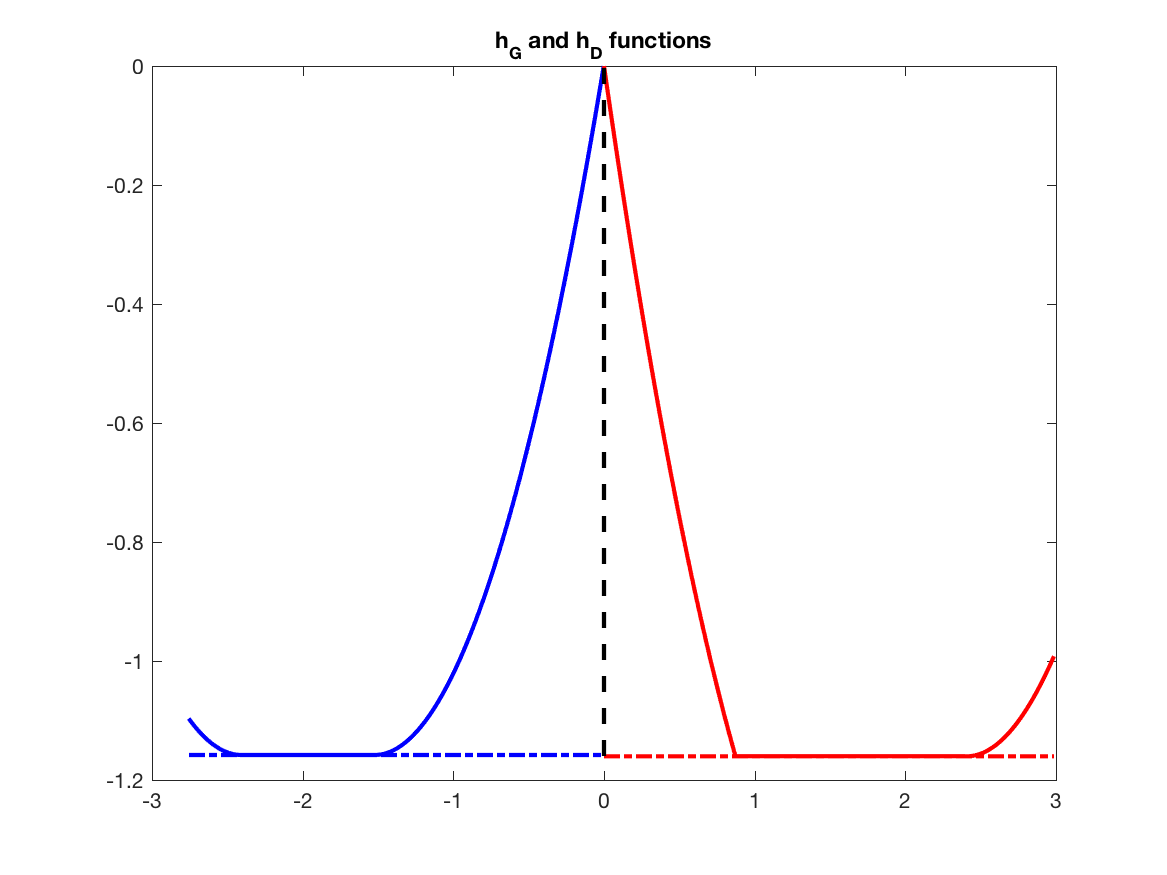}
  \includegraphics[trim=30 30 30 29,clip,width=0.48\linewidth]{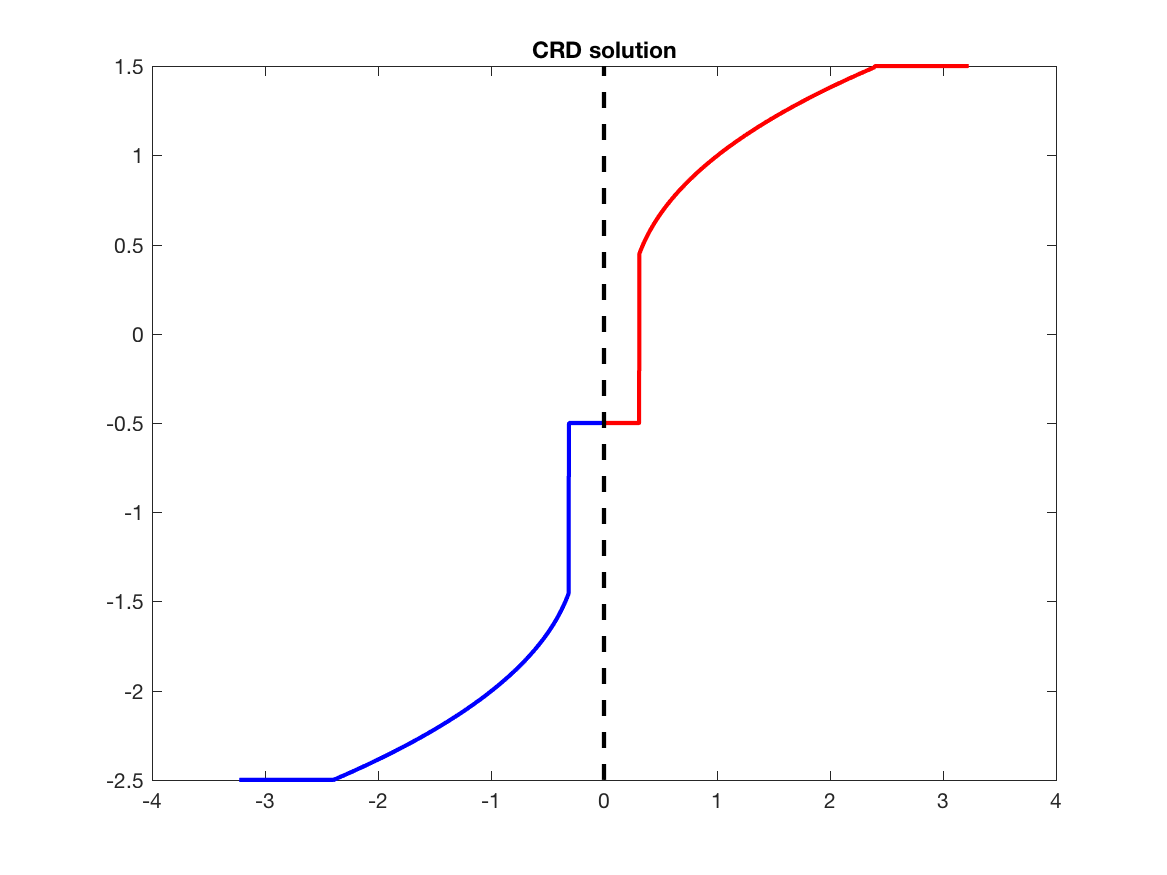}
  \includegraphics[trim=30 30 30 29,clip,width=0.48\linewidth]{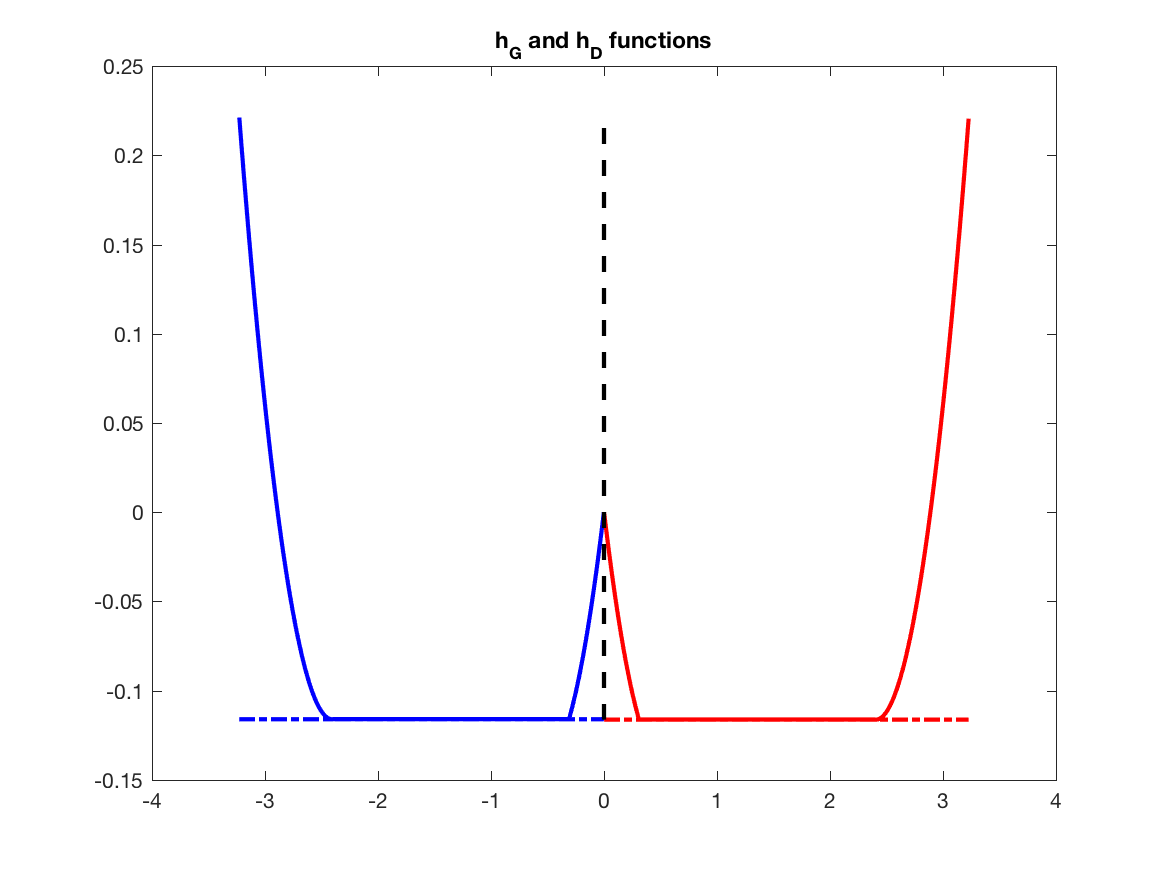}
  \includegraphics[trim=30 30 30 29,clip,width=0.48\linewidth]{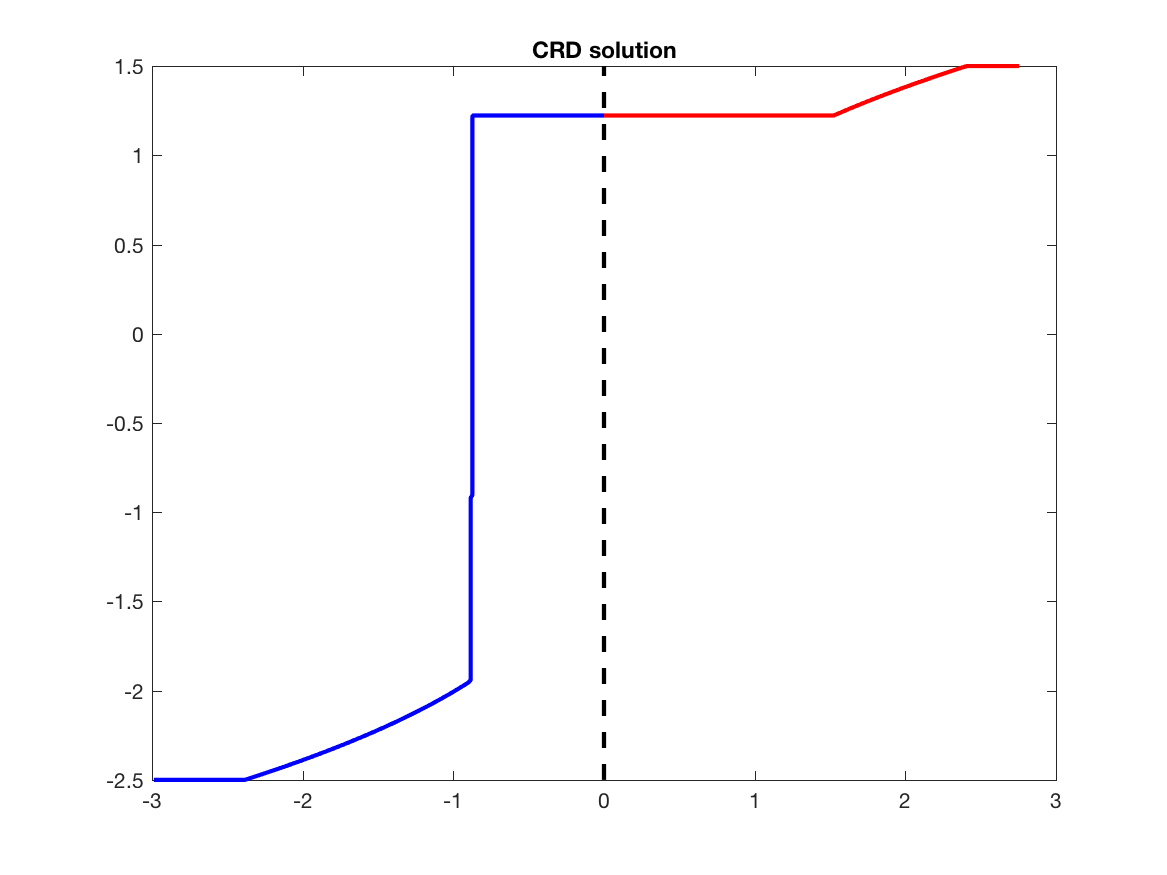}
  \includegraphics[trim=30 30 30 29,clip,width=0.48\linewidth]{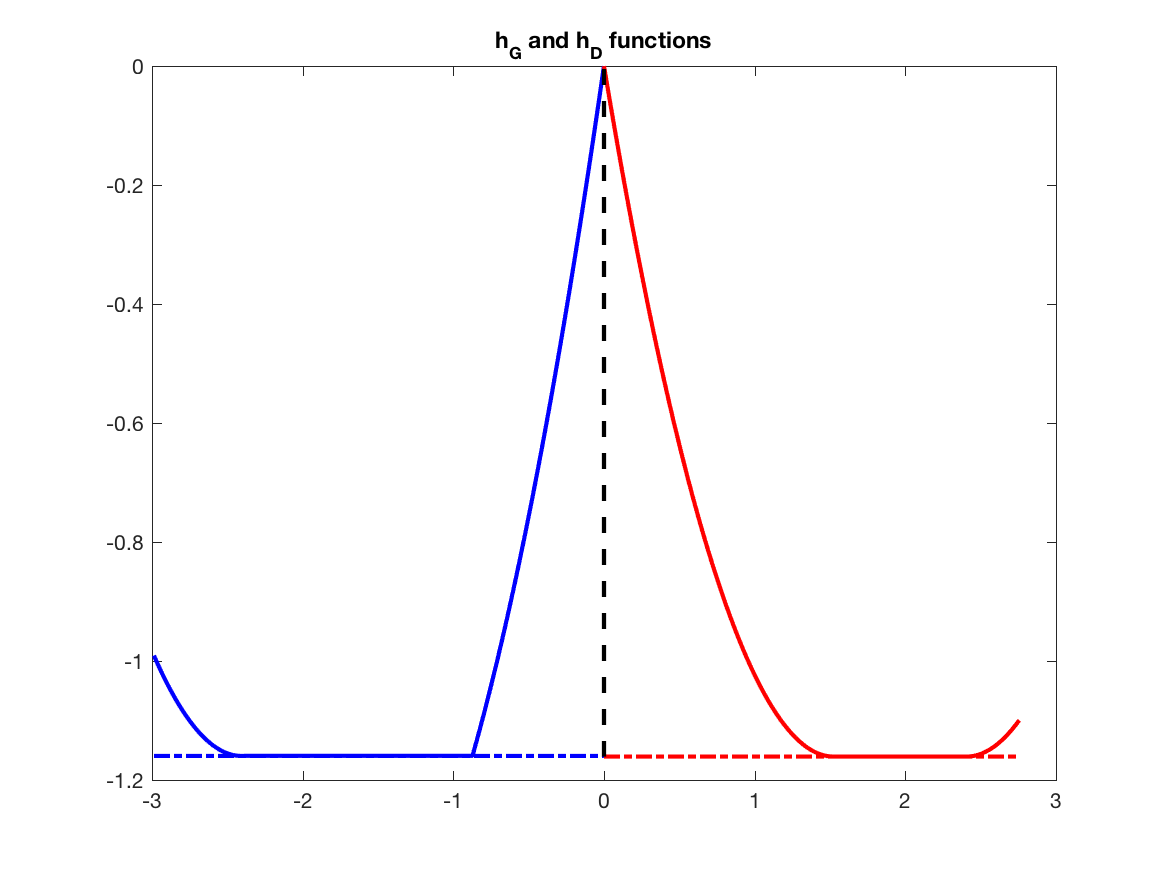}
  \caption{Three admissible solutions (from top to bottom). CRD solution $u$ (left) and selection criterion $h(\cdot;u)$ (right).}
   \label{fig:Multiple}
\end{figure}


\section{Concluding remarks}

From our results, it appears that Dafermos's regularization method applied to the nonconservative coupling problem is able to provide a partial selection of half-entropy weak solutions. A first selection process concerns the form of the internal coupling layers, which may be understood here as the ODE counterpart of the DLM path theory. When such waves are absent but left-hand and right-hand waves are present together, the solution $u$ being then continuous at the interface, the new partial global selection criterion is able to reduce the continuum of solutions to only a {\sl finite number of solutions.}
 In some more restricted situations, for instance for convex quadratic fluxes we have established the uniqueness of such solutions and made them explicit. Clearly, the nature of the regularization terms and of the coupling flux may strongly influence the value of these selected intermediate states and, therefore, the whole solution.


\section*{Acknowledgments} 

The three authors were partially supported by the Innovative Training Networks (ITN) grant 642768 (ModCompShock), and by the Centre National de la Recherche Scientifique (CNRS).


\bibliographystyle{amsplain}

\end{document}